\documentclass[11pt]{article}
\usepackage[centering, margin={0.8in, 0.5in}, includeheadfoot]{geometry}

\usepackage{amsmath,amssymb,amsthm,latexsym, amscd,epsfig,epstopdf,enumerate}
\usepackage{epsfig,graphicx,graphics,subfig}

\usepackage{amsmath}
\usepackage{float} 
\usepackage{subfig}
\usepackage{amsfonts}
\usepackage{epsfig,color,enumerate}
\usepackage{cases}
\usepackage{algorithmic}

\usepackage[hidelinks]{hyperref}

\usepackage{verbatim}
\usepackage{dsfont}
\usepackage{color}

\newtheorem{theorem}{Theorem}[section]
\newtheorem{definition}{Definition}[section]
\newtheorem{lemma}{Lemma}[section]
\newtheorem{prop}{Proposition}[section]
\newtheorem{corollary}{Corollary}[section]

\newtheorem{remark}{Remark}[section]

\def\R {{\mathbb R}}

\def\N {{\mathbb N}}
\def\E {{\mathbb E}}

\newcommand\dive{\mathop{\rm div}}

\graphicspath{{./pics/}}

\input{shortcuts.sty} %Makros, theorem enviroments
\usepackage{csquotes}

\newcommand{\tbf}{\textbf}

\newcommand{\tsl}{\textsl}
\newcommand{\mbb}{\mathbb}
\newcommand{\mbf}{\mathbf}
\newcommand{\mc}{\mathcal}

\newcommand{\veps}{\varepsilon}
\newcommand{\what}{\widehat}
\newcommand{\wtilde}{\widetilde}
\newcommand{\vphi}{\varphi}
\newcommand{\oline}{\overline}

\newcommand{\ra}{\rightarrow}

\newcommand{\g}{\gamma}
\newcommand{\s}{\sigma}

\newcommand{\de}{\delta}
\newcommand{\lan}{\langle}
\newcommand{\ran}{\rangle}

\renewcommand{\div}{{\rm div}\,}
\renewcommand{\det}{{\rm det}\,}

\allowdisplaybreaks

\def\d{\partial}

\usepackage{xcolor}

\numberwithin{equation}{section}

\begin{document}
\title{\LARGE\bf A theoretical investigation of Brockett's ensemble\\ optimal control problems}

\author{
Jan Bartsch{\footnote{\textsc{J. Bartsch}:
\texttt{jan.bartsch@mathematik.uni-wuerzburg.de}; Institut f\"ur Mathematik,
Universit\"at W\"urzburg,
Emil-Fischer-Strasse 30,
97074 W\"urzburg,
Germany.}}\and
Alfio Borz\`i{\footnote{\textsc{A. Borz\`i}:
\texttt{alfio.borzi@mathematik.uni-wuerzburg.de}; Institut f\"ur Mathematik,
Universit\"at W\"urzburg,
Emil-Fischer-Strasse 30,
97074 W\"urzburg,
Germany.}}\and
Francesco Fanelli{\footnote{\textsc{F. Fanelli}:
\texttt{fanelli@math.univ-lyon1.fr};
Universit\'e Claude Bernard Lyon 1,
CNRS UMR 5208, Institut Camille Jordan,
43 blvd. du 11 Novembre 1918,
F-69622 Villeurbanne cedex, France.}}\and
Souvik Roy{\footnote{\textsc{S. Roy}:
\texttt{souvik.roy@uta.edu}; Department of Mathematics, The University of Texas at 
Arlington, Mathematics, 411 South Nedderman Drive, 
Box 19408, 
Arlington, TX 76019-0408.}}
}
\date{\small\today}

\maketitle

%Alfio: please abstract impersonal, no we ....

\begin{abstract}
{\small This paper is devoted to the analysis of problems of optimal control of ensembles governed by the Liouville (or continuity) equation.
The formulation and study of these problems have been put forward in recent years by R.W. Brockett, with the motivation that 
ensemble control may provide a more general and robust control framework.

Following Brockett's formulation of ensemble control, a Liouville equation with unbounded drift function, and a class of cost 
functionals that include tracking of ensembles and different control costs is considered.
For the theoretical investigation of the resulting optimal control problems, a 
well-posedness theory in weighted Sobolev spaces is presented for the Liouville and transport equations. 
Then, a class of non-smooth optimal control problems governed by the Liouville equation 
is formulated and existence of optimal controls is proved. Furthermore, optimal controls are characterised 
as solutions to optimality systems; %involving the forward Liouville equation and its adjoint, a transport equation, and a variational inequality representing the optimality condition. 
such a characterisation is the key to get (under suitable assumptions) also uniqueness of optimal controls.
}
\end{abstract}



\section{Introduction}
The notion of ensemble control was proposed by R.W. Brockett in \cite{Brockett1997}, and further in \cite{Brockett2012,Brockett2007}, 
while considering the problem of a trade-off between the complexity of implementing 
a control strategy and the performance of the control system. For the former, 
Brockett discusses the concept of minimum attention control that results in costs 
of the control that involve a partial time-derivative of the control function. For the latter, he emphasizes the advantage of 
considering an ensemble of trajectories, which stem from a distribution of initial 
conditions, rather than individual trajectories. By these two consideration, Brockett 
concludes that the natural setting for investigating both aspects of the 
resulting control problem is by means of the Liouville (or continuity)
equation that governs the evolution of the ensemble of trajectories. 

The Liouville equation is a hyperbolic-type partial differential equation, often used to model the evolution of density functions representing the probability 
density of multiple trials of a single evolving ordinary differential equation (ODE in brief) system, or the 
physical (e.g. particle) density of multiple non-interacting systems. 
In both cases, the function of the dynamics of the ODE model appears as the drift coefficient 
of the Liouville equation. Therefore the problem of controlling a trajectory of a 
finite-dimensional dynamical system is lifted to the problem of controlling a continuum of dynamical systems with the same control strategy. Specifically, this setting 
results in the problem of determining a single closed- or open-loop controller, which applies to a particular system over an infinite number of repeated trials, or to steer a family of finite-dimensional
dynamical systems. As discussed by Brockett, this approach represents a new control framework that is able to address a number of issues as uncertainty in initial conditions and
the trade-off mentioned above. 

The purpose of this paper is to contribute to the development of Brockett's Liouville-based ensemble control with a theoretical investigation 
of a class of Liouville optimal control problems with unbounded coefficients and cost functionals that are formulated in terms of 
the density and of different control costs. Those control costs model various requirements on the control functions, including control constraints. 
To the best of our knowledge, there is no similar investigation available in the literature yet.
We will present more details about the formulation of our problem in Subsections \ref{sec-DynModelsLiouville} and \ref{sec-formulation}. Thereafter, we give a detailed overview of the main results
of the paper in Subsection \ref{sec-overview}.

%Alfio: I prefer detaile in place of 'quite precise' , I would avoid to use future form 'will' etc. 

For the time being, let us point out that our work focuses on two main issues. First of all, we study existence and regularity 
of solutions to a Liouville initial-value problem with a so-called control-in-the-coefficients, which is related to a differential model 
with a linear and bilinear control mechanism. In general, this latter model has the following 
structure:  $\dot x (t)\,=\, \oline a + \oline b \,u(t) + \oline c \, v(t) \, x(t)$, where $\oline a,\, \oline b, 
\, \oline c$ are given functions and $u$ and $v$ denote the linear and bilinear controls, respectively. As 
we show below, the right-hand side of this equation corresponds to the 
drift coefficient of the Liouville equation $a(t,x;u,v)\,=\,\oline a + \oline b \,u(t) + \oline c \, v(t) \, x$, which is unbounded for $x\in\R^d$. For this reason, we consider 
the less investigated problem of existence and regularity of solutions to a Liouville equation having a 
drift which has at most a linear growth in space at infinity: we give a self-contained presentation of the well-posedness theory (due to DiPerna and Lions, see \cite{DiPernaLions1989};
see also \cite{Crippa_PhD}) in Sobolev spaces $H^m$, for drifts
$a\, \in\, L^1\bigl([0,T];C^{m+1}(\R^d)\bigr)$ such that $ \nabla a\,\in\,L^1\bigl([0,T];C_b^{m}(\R^d)\bigr)$.
In addition, in order to address continuity and Fr\'echet differentiability properties of the 
control-to-state map $G:\,(u,v) \mapsto \rho$, we extend existence and uniqueness results to the framework of appropriately weighted Sobolev spaces; to the best of our knowledge,
such a well-posedness theory, which is natural in our context, seems to be new in the literature.

These results are essential to study ensemble optimal control problems which include Brockett's cost functionals with density of ensemble of trajectories, a $L^2$ cost of the 
control and a $H^1$ cost which promotes minimum attention controls. 
Moreover, we extend this framework including a $L^1$ cost of the control, which should 
promote minimum action control during the time evolution, and include 
box control-constraints. For the resulting Liouville optimal control problem, we 
discuss existence and uniqueness of optimal controls and their 
characterization by optimality systems. Specifically, we prove existence of 
controls in our general framework with weighted Sobolev spaces; in the case where only $L^2$ costs of the control are considered, we prove also their uniqueness.
For the characterisation of the optimal solution in our general setting, we use the (sub-)differentiability properties of 
the cost functional and of the control-to-state map and prove the 
first-order optimality conditions.
% PUT IN THE CORRRESPONDING SECETION: For clarity, we first discuss the case with $L^2$ costs only, then the case with $L^2-H^1$ costs, and finally the case  with $L^2-L^1- H^1$ costs.

We conclude this brief introduction by remarking that, according to Brockett, in the cost functionals of ensemble controls the density-based tracking and terminal observation terms
(i.e. respectively $\theta$ and $\vphi$ in \eqref{Jfunc} below) are 
formulated by introducing ``attracting'' potentials, which are quadratic (as discussed in \cite{Brockett2012,Brockett2007}).
Capturing such a framework represents an additional challenge for the theoretical investigation of these control problems: as shown in the final part of this paper, we are able to
address and solve this issue within our weighted Sobolev spaces framework.

\medbreak
The rest of the paper unfolds as follows. In Section \ref{sec-probformres}, we present the problem formulation and an overview of our main results. 
%Alfio: I would prefere a less discursive style: remove  (without reporting precise statements, though)
In Section \ref{sec-LiouvilleTransport}, we state and prove our theoretical results on well-posedness of Liouville and transport equations with unbounded drifts in weighted Sobolev spaces.
Section \ref{s:control-map} is devoted to the study of the Liouville control-to-state map $G$. Section \ref{s:ocp} is devoted to the analysis of the 
Liouville optimal control problem: we establish existence of optimal controls, their characterisation as solutions of a first-order optimality system, and by use of the previous
characterisation, their uniqueness in specific cases. An appendix completes this work, where we postpone the proof of some technical results.

%%%%%%%%%%%%%%%%%%%%%%%%%%%%%%%%%%%%
\subsubsection*{Acknowledgements}
%%%%%%%%%%%%%%%%%%%%%%%%%%%%%%%%%%%%%

The third author has been partially supported by the LABEX MILYON (ANR-10-LABX-0070) of Universit\'e de Lyon, within the program ``Investissement d'Avenir''
(ANR-11-IDEX-0007),  by the project BORDS (ANR-16-CE40-0027-01)
and  the programme ``Oberwolfach Leibniz Fellows'' by the Mathematisches Forschungsinstitut Oberwolfach in 2017.

\subsubsection*{Notation}
In this section, we present our notation that we use throughout the paper.

Given a domain $\Omega\subset\R^d$, the symbol $C_c^\infty(\Omega)$ denotes the space of infinitely often differentiable functions
with compact support in $\Omega$.
Given $k\in\N$, we denote by $C^k(\Omega)$ the space of all $k$-times continuously differentiable functions defined on $\Omega$, and by $C_b^k(\Omega)$ the subspace of $C^k(\Omega)$
formed by functions which are uniformly bounded together with all their derivatives up to the order $k$. We equip $C_b^k(\Omega)$ with the $W^{k,\infty}$-norm as follows 
\begin{align*}
\standardNorm{v}_{C^{k}_b}\, :=\, \sum_{|\alpha| \leq k} \standardNorm{D^\alpha v}_{L^\infty}\,.
\end{align*} 
For $\alpha\in\,]0,1]\,$, we denote with $C^{0,\alpha}(\Omega)$ the classical H\"older space (Lipschitz space if $\alpha=1$), endowed with the norm
\begin{align*}
\standardNorm{\Phi}_{C^{0,\alpha}}\,:=\,\sup_{x \in \Omega} |\Phi(x)|\, +\, \sup_{\substack{x,y \in \Omega \\ 0<|x-y|\leq1}}  \frac{|\Phi(x)-\Phi(y)|}{|x-y|^\alpha}\,.
\end{align*}
In particular, $C^{0,1}(\Omega)\,\equiv\,W^{1,\infty}(\Omega)$.

For $k \in \NN$ and $1 \leq p \leq +\infty$, we denote with $W^{k,p}(\Omega)$ the usual Sobolev space of $L^p$ functions with all the derivatives up to the order $k$ in $L^p$; we also set
$H^k(\Omega):=W^{k,2}(\Omega)$.
For $1\leq p < +\infty$, let $W^{-k,p}(\Omega)$ denote the dual space of $W^{k,p}(\Omega)$.
For any $p\in[1,+\infty]$, the space $L^p_{loc}(\Omega)$ is the set formed by all functions which belong to $L^p(\Omega_0)$, for any compact subset $\Omega_0$ of $\Omega$.

Furthermore, we make use of the so-called Bochner spaces. Given two Banach spaces $X$ and $Y$ and a fixed time $T>0$, we define
\begin{align*}
 X_T(Y)\, := \,X\bigl([0,T];Y\bigr)\,,\qquad\qquad \mbox{ with }\qquad \standardNorm{u}_{X_T(Y)}\,:=\,\int_0^T \standardNorm{u(t)}_Y\, dt\,.
\end{align*}

Given a Banach space $X$ and a sequence $\bigl(\Phi_n\bigr)_n$, we use the notation $\bigl(\Phi_n\bigr)_n \subset X$ meaning that $\Phi_n \in X$ for all $n \in \NN$ and
that this sequence is uniformly bounded in $X$: there exists some constant $M>0$ such that $\standardNorm{\Phi_n}_X \leq M~\forall n \in \NN$.

Given two Banach spaces $X$ and $Y$, the space $X\cap Y$, endowed with the norm $\|\cdot\|_{X\cap Y}\,:=\,\|\cdot\|_{X}\,+\,\|\cdot\|_{Y}$, is still a Banach space.

For every $p\in[1,+\infty]$, we use the notation $\LL^p_T(\RR^d)\,:=\, L^p_T(\RR^d) \times L^p_T(\RR^d)$. Analogously, $\HH^1_T(\RR^d)\,:=\, H^1_T(\RR^d) \times H^1_T(\RR^d)$. 
In addition, given two vectors $u$ and $v$ in $\R^d$, we write $u\leq v$ if the inequality is satisfied component by component by the two vectors: namely, $u^i\leq v^i$ for all $1\leq i\leq d$.

Given two operators $A$ and $B$, we use the standard notation $[A,B]$ for their commutator: $[A,B]\,:=\,AB-BA$.

\section{Problem formulation and overview of the results}
\label{sec-probformres}
In this section, we present the problem formulation and an 
overview of our results. In Section \ref{sec-DynModelsLiouville}, 
we discuss the Liouville equation and the control mechanism. In 
Section \ref{sec-formulation}, we formulate our 
Liouville ensemble optimal control problem. Section \ref{sec-overview} 
illustrates our main results. 
%Alfio: remove (with no rigorous statements, though). 

\subsection{The Liouville equation and a control mechanism}
\label{sec-DynModelsLiouville}

The Liouville model represents the fundamental 
building block of many important equations in fluid mechanics and related fields, as e.g. the Boltzmann 
equation, the Fokker-Planck equation, the Vlasov equation. It arises in  diverse areas of sciences 
as biology, finances, mechanics, and physics; 
see e.g. \cite{Cercignani1969,ChoVenturiKarniadakis2015,CockshottZachariah2013, OceanDynamics, EisemanStone1980,
Colonna2016,Risken1996}.

Central to our discussion is the relation of the Liouville equation 
to a given ODE evolution model. To illustrate this fact, consider a smooth vector field $a(t,x)$ over $\R^d$, where $(t,x)\in[0,T]\times\R^d$, for some time $T>0$. We refer to $a$ as the 
drift function. It is well-known that, if a scalar function $\rho$, defined on $[0,T]\times\R^d$, 
satisfies the Liouville equation
\begin{equation}
\d_t\rho(t,x)\, +\, \div\bigl( a(t,x) \, \rho (t,x)\bigr)\, =\,0, 
\label{Liouville1D}
\end{equation}
with some (say) smooth initial datum $\rho_{|t=0}\,=\,\rho_0$, then we can represent $\rho$ by the formula
$$
\rho(t,x)\,=\,\frac{1}{\det J(t,x)}\,\rho_0\bigl(\psi_t^{-1}(x)\bigr)\,,
$$
where $\psi_t(x)\,=\,\psi(t,x)$ denotes the flow map associated to $a$, $J(t,x)\,=\,\nabla_x\psi_t(x)$ is its Jacobian matrix, and $\psi_t^{-1}(x)$ means the inverse with respect to the space variable,
at $t$ fixed. By definition of flow map, $\psi$ verifies the following system of ODEs
\begin{equation}
\d_t\psi(t,x)\,=\,a\bigl(t,\psi(t,x)\bigr)\,, \qquad \psi(0,x)=x\,.
\label{ODE1}
\end{equation}
Notice that we can equivalently write this Cauchy problem 
as $\dot{y}(t)=\,a\bigl(t,y(t)\bigr)$, $y(0)=x$, to point out that 
the (independent) space variable in the Liouville equation 
corresponds to the (dependent) state variable of the 
related dynamical system. In the following, we shall use the 
same symbol for both cases. 

Problem \eqref{Liouville1D} models the evolution of the ensemble of trajectories of \eqref{ODE1} for a density  distribution of initial conditions given by $\rho_0$. 
Thus, one possible interpretation of \eqref{Liouville1D} is that the function $\rho(t,x)$ represents the probability density function of finding the system in $x$ at time $t$,
assuming that $\rho_0$ prescribes the initial probability density for $\psi_t^{-1}(x)$. 
This interpretation in the space of probability appears  in the realm of Kolmogorov-Fokker-Planck and Kramers equations related to stochastic processes, where 
the Liouville equation corresponds to the first-order differential term of these equations; see e.g. \cite{Risken1996}. On the other hand, if $\rho$ represents the material density
of non-interacting particles, then \eqref{Liouville1D} models the evolution of this density, see e.g. \cite{Fei-No}

Notice that the former point of view is the predominant one 
in Brockett's consideration and we adopt it in the rest of the paper. Therefore, it is natural to assume an initial condition for the 
Liouville model such that $\rho_0\ge 0$, together with the normalization $\int_{\R^d} \rho_0(x)dx=1$. 
%Alfio: we will keep in mind this point of view
Consequently, by equation \eqref{Liouville1D}, it holds
for all times $t\geq0$ that
$$
\rho(t,x)\,\ge\,0\qquad\mbox{ and } \qquad \int_{\R^d} \rho(t,x)\,dx\,=\,\int_{\R^d} \rho_0(x)dx\,=\,1\,.
$$
The first property can be proved by the vanishing viscosity method 
and the maximum principle, see e.g. \cite{Fei-No,GodlewskiRaviart1991}, or just by solving \eqref{Liouville1D} along characteristics; the second 
property follows from a simple application of the divergence theorem. 
%Alfio:
However, notice that most of our results do not require
the latter two assumptions on $\rho_0$.

\medbreak
Next, let us discuss our control mechanism. We remark that the focus of ensemble control is the development of a control strategy for the 
differential model \eqref{ODE1} augmented with a 
control mechanism, as follows
\begin{equation}
\dot{x} = a(t,x;u),
\label{ODE2}
\end{equation}
where $u$ denotes the control function. 

We refer to \cite{Brockett2012,Brockett2007} for a 
discussion on the choice of $u$ as a function of time only, 
which corresponds to a so-called open-loop control, 
or as a function of time and of the state variable, which
may represent a feedback law. In this paper, while we consider our controlled Liouville 
model in a general setting that accommodates both 
choices, we focus our attention on open-loop optimal control problems: this point of view is
motivated by the fact that the most used 
control mechanisms for \eqref{ODE2} are the linear and 
bilinear ones, as follows,
\begin{equation}
a(t,x;u)\,=\,a_0(t,x)\,+\,u_1(t)\,+\, x \circ u_2(t)\,, %\qquad u=(u_1,u_2),
\label{controlmechanism}
\end{equation}
where $a_0$ is a smooth vector field and $u\,=\,(u_1,u_2)$ is the control, which, for the scope of the present discussion, we assume to be smooth. The control $u_1$ represents a linear 
control mechanism and $u_2$ multiplying the state 
variable $x$ represents the bilinear control term. 
Both functions $u_1$ and $u_2$ are defined on the time interval $[0,T]$ with values in $\R^d$. Further, with $\circ:\R^d \times \R^d \rightarrow \R^d$ we denote the Hadamard product
of two vectors, i.e. the multiplication component by component.

Now, notice that corresponding to the controlled 
evolution model \eqref{ODE2}, we have the following controlled 
Liouville equation
\begin{equation}
\d_t\rho(t,x)\, +\, \div\bigl( a(t,x;u) \, \rho (t,x)\bigr)\, =\,0. 
\label{Liouville1Dcontrol}
\end{equation}
In this framework and in the simple case $a_0=0$ and $d=1$, we can 
give a simple interpretation of the role of $u_1$ and $u_2$: The control 
$u_1$ represents the driving force of the mean value of the density; $u_2$ determines the evolution of the 
variance of the density. In fact, let 
$\rho_0$ represents a normalized Gaussian centred at $x_0$ with variance $v_0$ at time $t=0$, and define the following average operator 
$$
\E[g](t)\, = \,\int_{\RR} g(x) \, \rho(t,x) \, dx \,.
$$
In particular, we have the mean $m(t)=\E[x](t)$ and the variance $v(t)=\E\big[\big(x-m(\cdot)\big)^2\big](t)$. Then, by taking the average of our controlled ODE model, we obtain the following equations:
\begin{align*}
\dot m (t)\, =&\, u_1(t)\, +\, m(t) \, u_2(t) , \qquad &m(0)\, =\,x_0\,, \\
\dot v (t)\, =&\, 2 \, v(t) \, u_2(t) \,, \qquad &v(0)\,=\,v_0\, .
\end{align*}
However, because of the limiting assumptions, this construction does not 
provide the degree of generality that the Liouville framework offers, as the 
latter allows to accommodate any chosen drift with any control mechanism. 
The Liouville framework is also very attractive because it allows to 
consider bimodal distributions of the initial density function. 

Finally, we remark that, for the characterization of the solution to our Liouville optimal control problems, we shall deal with \eqref{Liouville1Dcontrol} and with an adjoint
Liouville problem, namely a transport problem, given by 
\begin{equation} \label{eq:transp}
\d_tq(t,x)\,+\,a\bigl(t,x;u\bigr)\cdot \nabla q(t,x)  \,=\,g(t,x)\,, \qquad \mbox{ with }\quad q_{|t=0}\, =\,q_0 \,,
\end{equation}
where $g$ and $q_0$ depend on the optimization data.

\subsection{Formulation of ensemble optimal control problems} \label{sec-formulation}
In order to discuss the formulation of ensemble control, 
consider the following ODE optimal control problem
\begin{align}
&\min j(x,u)\,:=\, \int_0^T \Big( \theta\big(x(t)\big)\, +\, \kappa\big(u(t)\big) \Big) \,dt\, +\, \varphi\big(x(T)\big) \label{jfunc}\\
&\mbox{s.t.} \qquad \dot{x}(t)\,= \,a\big(t,x(t);u(t)\big)\,, \qquad x(0)\,=\,x_0\,, \label{ODEcontrol}
\end{align}
where ``s.t.'' stands for ``subject to''. Here above, $\theta$, $\kappa$ and $\varphi$ are usually taken to be continuous convex functions of their arguments; we will better specify their properties
later on in the present section.

Further, assume that the optimal control function $u$ is sought in the following set of admissible controls:
\begin{equation}
U_{ad}\,: =\, \left\{ u \,\in\, \LL^\infty_T(\RR^d)\;\bigl|\quad u^a\, \leq\, u(t)\, \leq\, u^b \qquad\mbox{ for a.e. }\; t\,\in\,[0,T]\right\}\,. 
\label{setUad}
\end{equation}
In particular, in the case of \eqref{controlmechanism}, 
we have two box constraints $u^a\,=\,(u_1^a,u_2^a)$ and $u^b\,=\, (u_1^b,u_2^b)$, where $u_j^a < u_j^b$, $j=1,2$, are given vectors in $\R^{d}$. Clearly, the optimal control function $u$ that 
solves \eqref{jfunc}-\eqref{ODEcontrol} with $u \in U_{ad}$ depends on the initial condition $x_0$, which is fixed, and it represents a control strategy that 
is determined once and for all times for the given $x_0$ 
and the given optimization setting. Therefore any uncertainty 
on the initial condition is not taken into account in the 
formulation \eqref{jfunc}-\eqref{ODEcontrol} and, hence, 
the resulting control is not robust. On the other hand, a 
closed loop control, say, $u=u(t,x)$, would appropriately 
control the system based on the actual state of the system; however, as pointed out in \cite{Brockett2007}, the cost of implementing such a control mechanism 
is often prohibitive and may be not justified by real applications. 

For this reason, with the purpose to strike a balance 
between the desired performance of the system and the cost 
of implementing an effective control, the ensemble control strategy considers a density of initial conditions, and therefore ensemble of trajectories. In this way, it aims at achieving
robustness, while choosing control costs which promote controls allowing for easier implementation (see below). 
Thus, one is led to the formulation of the following 
ensemble optimal control problem 
\begin{align}
&\min_{u \in U_{ad}} J(\rho,u)\,:=\,\int_0^T \int_{\R^d} 
\theta(x)  \, \rho(x,t) \, dx\, dt\, +\, \int_{\R^d} \varphi(x) \, \rho(x,T) \, dx\, +\, \int_0^T \kappa\big(u(t)\big) \, dt
\label{Jfunc}\\
&\mbox{s.t.} \qquad \d_t\rho\, +\, \dive \bigl( a(t,x;u) \, \rho \bigr)\, =\,0\,, \qquad \rho_{|t=0}\,=\, \rho_0 \,. 
\label{LiouvilleEQ}
\end{align}
This problem is defined on the space-time cylinder $\RR^d \times [0,T]$, for some $T >0$ fixed. 
In this formulation, the initial density $\rho_0$ represents the probability 
distribution of the initial condition $x_0$ in \eqref{jfunc}-\eqref{ODEcontrol}, and 
thus it models the known uncertainty on the initial data.

\medbreak
Next, we discuss some specific choices of the optimization components in \eqref{jfunc}-\eqref{ODEcontrol}, and correspondingly in 
\eqref{Jfunc}-\eqref{LiouvilleEQ}.

For example, if 
$x=0$ is a critical point for \eqref{ODEcontrol}, which requires $a(t,0;u)=0$, then the choice $\theta(x)=x^2$  appears standard for stabilization purposes. Usually, in this 
context, the so-called $L^2$ cost of the control is considered, 
which corresponds to the choice $k(u)=\gamma \, u^2$, where 
$\gamma >0$ is the weight of the cost of the control. On the other hand, if the 
purpose of the control in \eqref{jfunc}-\eqref{ODEcontrol} 
is to track a desired and even non-attainable trajectory $x_d \in L^2(0,T;\R^d)$, and to come close to a given final configuration $x_T \in \R^d$ at the final time (possibly with
$x_d(T) \neq x_T$), then a natural choice 
appears to be $\theta\big(x(t)\big)=\alpha\big(x(t) - x_d(t)\big)^2$ and $\varphi\big(x(T)\big)=\beta\big(x(T)-x_T\big)^2$, with appropriately chosen weights $\alpha, \,\beta >0$. However, in the 
context of ensemble control, as in \eqref{Jfunc}, the choice of $\theta$ and $\varphi$ 
as convex functions is problematic because of integrability issues. On the other 
hand, we remark that the role of this function is to define an attracting potential, 
that is, to define a well centred at a minimum point such that the minus gradient 
of the potential is directed towards this minimum. 
%Alfio: for this purpose
For this purpose, a possible choice is $\theta(x)=1-\exp(-x^2)$,
with the minimum at $x=0$. In our analysis, we are able to address both cases in the framework of weighted Sobolev spaces: the case of attracting potentials $\theta$ and $\varphi$ which
are both $L^2$ integrable, and the case of $\theta$ and $\varphi$ which are quadratic functions.
Notice that, in any case, the modelling choice for \eqref{jfunc}-\eqref{ODEcontrol} translates without changes to \eqref{Jfunc}-\eqref{LiouvilleEQ}. 

As discussed in \cite{Brockett2012,Brockett2007, Brockett1997}, the choice of the cost function $\kappa$ should be such that the effort of implementing the control strategy 
is as small as possible. In this sense, the cost of implementing a slowly varying control function, and (we add) a control that does not act for all times, should 
be smaller than that corresponding to a control having large variations.  From this perspective, a constant input that controls the system 
is the cheapest choice, and the next possible choice is a control that slowly changes in time. 
This requirement leads naturally to a cost of the form
$$
\nu \, \int_0^T  \left( \frac{d u}{d t}(t) \right)^2 \, dt ,
$$
where $\nu \ge 0$ is a non-negative weight. In fact, as $\nu$ is taken larger, the 
resulting optimal control will have smaller values of its time derivative, that is, a 
slowly varying control, which is called ``minimum attention control'' in \cite{Brockett1997}. 

On the other hand, the most common quantification of the cost of a control in the 
context of \eqref{jfunc}-\eqref{ODEcontrol} is the $L^2$-cost already mentioned above.  
More recently, there has been a surge of interest in $L^1$-costs, originating 
from signal reconstruction and magnetic resonance imaging \cite{Candes2006}. 
This cost is given by 
$$
\delta \, \int_0^T  \left| u(t) \right| \, dt ,
$$
where $\delta \ge 0$. The effect of this cost is that it promotes sparsity of the control function, in the 
sense that, as $\delta >0$ is increased, the $u$ resulting from the 
minimisation procedure will be zero on open intervals in $\,]0,T[\,$, 
and these intervals become larger and eventually cover all of $\,]0,T[\,$ as $\delta \to +\infty $. 
In the present papaer, we introduce the $L^1$-cost in the context of ensemble control and call the 
resulting sparse control a ``minimum action control''. 

All together, we specify the term $\int_0^T \kappa\big(u(t)\big)\,dt$ in \eqref{jfunc} and in 
\eqref{Jfunc} as follows
\begin{equation}
\kappa\big(u(t)\big)\, :=\, \frac{\gamma}{2} \,   \big( u(t) \big)^2\, +\, \delta \,   \left| u(t) \right| \, +\, \frac{\nu}{2} \,  \left( \frac{d u}{d t} (t)\right)^2 \, ,
\label{controlcosts}
\end{equation}
where $\gamma + \delta + \nu > 0$ and the factor $1/2$ is chosen for convenience 
of later calculations. 

Notice that different choices of the value of the weights $\gamma, \, \delta, \, \nu$ 
will result in different features of the resulting optimal control function. We investigate the properties of $u$ resulting from these different choices in 
Section \ref{sec-existencecontrol} and its characterisation in Section 
\ref{sec-OptimalitySystem}. 

\subsection{Overview of the main results}
\label{sec-overview}
%Alfio: less discursive, no will , no here, no folklore :-) 
In this section, we give a short overview of the main results contained in this paper.

The first step in our analysis consists in investigating the well-posedness of the PDE under consideration and its adjoint, namely continuity 
and transport equations with unbounded drift function, which presents the structure \eqref{controlmechanism}. We perform such a study in Section \ref{sec-LiouvilleTransport} below.
%
%The well-posedness theory for Liouville and transport equations under minimal regularity assumptions on the drift
%vector field $a$ has attracted the attention of many scientists in the last decades.

We point out that we do not strive for minimal regularity hypotheses on the drift vector field $a$, and frame our work within a setting that can be considered (for bounded coefficients)
classical; see e.g. Chapter 3 of \cite{BCD}.
Nonetheless, to the best of our knowledge very few references deal with
the case of unbounded drifts, with at most linear growth at infinity, for which the well-posedness theory goes back to the conrnerstone paper \cite{DiPernaLions1989}
(see also \cite{Amb_2004, AmbrosioCrippa2008, Amb-Cr_2014, Crippa_PhD} and references therein). 
For this reason, for clarity, we give a self-contained presentation of the well-posedness theory for the Liouville and transport equations with unbounded drifts in classical Sobolev spaces:
this part is just a review, and contains no new results.
%Alfio: non e' chiaro se questi lavori coprano gia' il nostro caso unbounded oppure no. 
% non si capisce

%Alfio: remove in a while, will 
In addition, for reasons which will become clear later, in Section \ref{ss:weight} we investigate well-posedness of the Liouville and transport equations in weighted
spaces $H^m_k$, see Definition \ref{def:H^m_k} below. Roughly, a tempered distribution
%Alfio: distribution - spell check ?
$\rho\in H^m$ belongs to $H^m_k$ if it enjoys further integrability at $+\infty$,
namely if $\rho$ and all its derivatives up to order $m$ belong to the measurable space $\big(L^2(\R^d),(1+|x|)^k\,dx\big)$.
On the one hand, existence, uniqueness and regularity properties are derived in this context by standard arguments: the key of the analysis reduces to show suitable \tsl{a priori} estimates
on the solutions in weighted norms. On the other hand, the results of this part seem to be completely new in the literature; in addition, this framework reveals to be well-adapted to the
investigation of our ensemble optimal control problem, see more details below.

%Alfio: remove a couple of remarks 
We believe that well-posedness in weighted spaces can be adapted with no special problems to $L^p$-based spaces,
for any $1\leq p<+\infty$. Then, ensemble optimal control problems with different integrability conditions can be considered as well.
In addition, very likely the $H^m_k$ theory should generalise also to more general hyperbolic systems which are symmetrizable in the sense of Friedrichs (see e.g. \cite{BG-Serre_2007, Met_2008}).
Consequently, we expect to be able to deal with optimal control problems related to those systems by similar methods as the ones presented in this paper.
However, extensions of the present work to both directions go beyond the scope of our paper, and we leave them for further studies.

\medbreak
After establishing well-posedness of the Liouville equation in a suitable framework, we pass to investigating the optimal control problem related to that PDE. 
For this, we follow a standard scheme.

First of all, in Section \ref{s:control-map} we define and study the Liouville control-to-state map $G$, namely the map which associates to any control state $u$
the unique solution $\rho\,=\,G(u)$ to the corresponding Liouville equation. A fundamental issue in this part is to prove Fr\'echet differentiability (in a suitable topology) of $G$.
Our method to get that property (see Section \ref{ss:G-diff}) consists in applying the definition of Fr\'echet differentibility, and showing the convergence
of the limit under consideration in the strong $L^\infty_T(L^2)$ topology, by performing fine stability estimates on the Liouville equation. Now, dealing with the growth 
in space of our drift
function at $+\infty$ requires the use of weighted norms and weighted spaces in order to carry out those estimates; moreover, due to the hyperbolicity of
transport and continuity equations, a loss of regularity occurs, which requires to consider both higher smoothness and higher integrability on the initial data
(namely, both $m\geq 2$ and $k\geq 2$). Fr\'echet differentiability reveals to be a fundamental property for the subsequent analysis.

In Section \ref{s:ocp}, we complete the investigation of the ensemble optimal control problem. First of all, we prove (see Theorem \ref{thm:existenceOptimalSolutions})
existence of optimal controls, in the simpler case of attracting potentials which are moreover $L^2$. Then, we characterise these optimal controls as solutions
of a first-order optimality system, whose interpretation in terms of the Fr\'echet differential of the reduced functional $\what J(u)\,:=\,J\big(u,G(u)\big)$ exploits the Fr\'echet differentiability
of $G$, proved above.
%Alfio: remove here
We also notice that the differentiability properties of $J$ change radically depending 
on the choice of the optimization weights. For instance, if $\gamma >0$ and $\delta=\nu=0$, then the optimization space is $L^2(0,T)$ and we have 
Fr\'echet differentiability of the cost functional. This is the ``standard'' case. 
If instead $\delta > 0$, then we have a semi-smooth optimal control problem and we have to resort to the use of 
sub-differentials. Finally, if $\nu >0$, then $H^1(0,T)$ is the appropriate control space,
and the optimality condition accounts for this fact. If all weights 
are positive and with control constraints, we have an optimal control problem whose structure (to the best of our knowledge) has never 
been investigated in PDE optimization. For this general case, we prove (see Theorem \ref{thm:OptimalitySystemSSN}) existence of Lagrange  multipliers, and derive the optimality system.

%Alfio: may be too many 'first of all'
Finally, in Section \ref{ss:unique} we address the uniqueness of optimal ensemble controls, in the special case $\g>0$ and $\delta=\nu=0$.
This part of the analysis exploits in a fundamental way the optimality system previously derived, and the characterisation of optimal controls as solutions to it.
In Theorem \ref{thm:opt-contr_u} we prove uniqueness of the optimal control in the control-unconstrained case; for this, we neeed some further assumptions
on the attracting potentials, namely we suppose $\theta$ and $\varphi$ to belong both to $H_1^1(\R^d)$. Notice that, to the best of our knowledge, our uniqueness result 
has no counterpart in the literature of PDE control problems. 
In the case that control constraints are present, in Theorem \ref{thm:opt_u-constr} we prove uniqueness of optimal controls, provided a smallness condition is satisfied.
Such a condition requires the time $T$ and the size of the data $\rho$, $g$, $\theta$ and $\vphi$ to be small enough, or the coefficient $\g$ to be sufficiently large.

We recall that all the results of Section \ref{s:ocp} that we have summarised above are obtained assuming $\theta, \, \varphi \in L^2(\R^d)$. However, in Section \ref{ss:confining} 
we show the necessary modifications to be performed in our arguments in order to address the case $\theta(x)=|x|^2$ and $\varphi(x)=|x|^2$.
%Alfio: shorten
In particular, such modifications require to solve the transport equation in weighted Sobolev spaces of negative index, see Lemma \ref{l:transp-neg} below.

\section{Theory of Liouville and transport equations with unbounded drifts} \label{sec-LiouvilleTransport}

In this section, we present results concerning the well-posedness theory of Liouville and transport equations in the class of Sobolev spaces.
In view of formula \eqref{controlmechanism}, we are especially interested in the case when the drift function $a$ may be unbouded, but has at most a linear growth at infinity.

In Subsection \ref{ss:classical}, we review the well-posedness theory in classical $H^m$ spaces, for $m\in\N$ (for simplicity); we will give
a self-contained presentation in this framework, and refer to e.g. \cite{DiPernaLions1989} and \cite{Crippa_PhD}
for details and more general results.
Afterwards in Section \ref{ss:weight}, motivated by the study of our optimal control problem, we extend these results to weighted Sobolev spaces. 
Notice that, although the statements and their proofs follow the main lines of the classical framework, 
to the best of our knowledge they are novel.

\subsection{Classical theory of Liouville and trasport equations} \label{ss:classical}

We start our discussion by considering the Liouville equation. Notice that 
our statements can be repeated in a very similar way (with just slight modifications) also for the adjoint Liouville problem, namely the transport equation: we treat this case in Paragraph \ref{sss:transport}, without giving details.

\subsubsection{Liouville equations in classical Sobolev spaces} \label{sss:H^m}
Consider the following Liouville initial-value problem 
\begin{equation} \label{eq:LiouvilleUnb} 
\left\{\begin{array}{ll}
        \d_t\rho\, +   \dive \bigl( a(t,x) \, \rho \bigr)\, = \, g(t,x) \qquad\qquad & \mbox{ in } \qquad [0,T]\times\R^d \\[1ex]
        \rho_{|t=0}\,=\, \rho_0 \qquad\qquad & \mbox{ on } \qquad \R^d\,.
       \end{array}
\right.
\end{equation}
%\begin{remark} 
Whenever attempting at solving equation \eqref{eq:LiouvilleUnb}, we have in mind its weak formulation. Namely, for all $\phi\, \in\, C_c^\infty\bigl(\RR^d \times [0, T[\,\bigr)$,
we want to verify the following equality
\begin{align} 
-\int_0^T\int_{\RR^d}\rho\,\partial_t \phi\,dx\,dt\,-\,\int_0^T \int_{\RR^d}\rho\, a\cdot\nabla\phi\,dx\,dt\,=\,
\int_0^T \int_{\RR^d}g\,\phi\,dx\,dt\, +\, \int_{\RR^d} \rho_0\,\phi(0)\,dx\,.
\label{eq:LweakUnb}
\end{align}
%\end{remark} 
The theory for this equation is classical, at least in the case of a bounded drift function $a$. 
The following well-posedness result is adapted to our needs from Theorem 3.19 in \cite{BCD}.
\begin{theorem}\label{thm:ex-u_L}
Let us fix $T>0$ and $m \in \NN$, and let $a \in L^1\bigl([0,T];C_b^{m+1}(\RR^d)\bigr)$, $\rho_0 \in H^m(\RR^d)$ and $g \in L^1\bigl([0,T];H^m(\RR^d)\bigr)$.

Then there exists a unique weak solution $\rho$ to \eqref{eq:LiouvilleUnb}, with $\rho\,\in\,C\bigl([0,T];H^{m}(\RR^d)\bigr)$.
Moreover, there exists a ``universal'' constant $C>0$, independent of $\rho_0$, $a$, $g$, $\rho$ and $T$, such that the following estimate holds true for any $t\in [0,T]$ : 
\begin{align*} %\label{est:en-est}
\standardNorm{\rho(t)}_{H^m}\, \leq \,C
\left(\standardNorm{\rho_0}_{H^m}\, +\, \int_0^t \standardNorm{g(\tau)}_{H^m}\, d\tau \right)\;\exp\left(C\,\int_0^t\standardNorm{\nabla a(\tau)}_{C^{m}_b}\,d\tau \right)\,.
\end{align*}
\end{theorem}

\begin{remark} \label{r:a}
In the case $m=0$, one can replace $\standardNorm{\nabla a}_{C^{0}_b}$ with $\|\div a\|_{L^\infty}$ inside the integral in the exponential term.
\end{remark}

Motivated by the study of our optimal control problem, see Section \ref{sec-DynModelsLiouville} and especially Definition \eqref{controlmechanism},
we are rather interested in the case when $a$ may be unbounded, with at most a linear growth at infinity. More precisely, given $m\in\N$, we assume
\begin{align}  \label{hyp:data} 
\begin{cases}
g \,\in \,L^1\bigl([0,T];H^m(\RR^d)\bigr) \quad\mbox{ and } \quad  \rho_0 \,\in\, H^m(\RR^d)
 \\[1ex]
a\, \in\, L^1\bigl([0,T];C^{m+1}(\R^d)\bigr)\,,\qquad\mbox{ with }\quad \nabla a\,\in\,L^1\bigl([0,T];C_b^{m}(\R^d)\bigr)\,. %\\ 
%\standardNorm{a(x,t)} \leq c(1+\standardNorm{x})~for~a.a. ~t\in [0,T]~\forall x \in \RR^d ,
\end{cases}
\end{align}

\begin{remark} \label{r:unb-drift}
Notice that hypotheses \eqref{hyp:data} imply, in particular, that $a(\cdot,\cdot)$ has at most linear growth in space at infinity: for almost every $(t,x)\,\in\,[0,T]\times\R^d$,
one has
$$
\left|a(t,x)\right|\,\leq\,C\,c(t)\,(1+|x|)\,,\qquad\qquad\mbox{ for }\qquad c\,=\,\left\|\nabla a\right\|_{L^\infty}\,\in\,L^1\bigl([0,T]\bigr)\,.
$$
The condition of at most linear growth at infinity can be proved to be somehow sharp for well-posedness, see e.g. \cite{DiPernaLions1989}, \cite{Crippa_PhD} and the references therein.
\end{remark}

The main result of this section is the following statement, proved by DiPerna and Lions in \cite{DiPernaLions1989} (see also \cite{Crippa_PhD}). However, we give here a self-contained presentation
of its proof.
\begin{theorem}
\label{thm:existenceUnboudedA}
Let $T>0$ and $m\in\N$ fixed, and let $a$, $\rho_0$ and $g$ satisfy hypotheses \eqref{hyp:data}.

Then there exists a unique solution $\rho\,\in\,C\bigl([0,T];H^{m}(\RR^d)\bigr)$ to problem \eqref{eq:LiouvilleUnb}.
Moreover, there exists a ``universal'' constant $C>0$, independent of $\rho_0$, $a$, $g$, $\rho$ and $T$, such that the following estimate holds true for any $t\in[0,T]$:
\begin{align} \label{est:en-estUnb}
\standardNorm{\rho(t)}_{H^m}\, \leq \,C
\left(\standardNorm{\rho_0}_{H^m}\, +\, \int_0^t \standardNorm{g(\tau)}_{H^m}\, d\tau \right)\;\exp\left(C\,\int_0^t\standardNorm{\nabla a(\tau)}_{C^{m}_b}\,d\tau \right)\,.
\end{align}
\end{theorem}

We notice that, also in this case, Remark \ref{r:a} applies.

\medbreak
The rest of this paragraph is devoted to the proof of Theorem \ref{thm:existenceUnboudedA}.
We proceed by truncation, and approximate our problem with a family of Liouville equations with bounded coefficients,
to which the classical theory (see Theorem \ref{thm:ex-u_L} above) applies. Then, we pass to the limit in the approximation parameter, proving convergence to a solution of the original equation.
We conclude by discussing time regularity and uniqueness issues.

\paragraph{Existence.}
The first step is to construct a suitable 
truncation of the drift function. For this purpose,  let us introduce a smooth cut-off function $\chi \in C^\infty_c(\RR^d)$ such that $\chi$ is radially decreasing,
$\chi(x)=1$ for $|x|\leq1$ and $\chi(x)=0$ for $|x|\geq2$. For all real $M>0$, we define
\begin{align}
a_M(t,x)\,:=\, \chi\Big( \frac{x}{M} \Big)\,a(t,x)\,.
\label{eq:aM}
\end{align}
Notice that, by assumptions \eqref{hyp:data}, we immediately get that $a_M\,\in\,L^1_T(C^{m+1}_b)$ for all $M>0$. Moreover, in view of Remark \ref{r:unb-drift},
an easy computation shows that 
\begin{equation} \label{ub:Da_M}
\bigl(\nabla a_M\bigr)_M\,\subset\,L^1_T(C^m_b)\,,\qquad\qquad\mbox{ with }\qquad \left\|\nabla a_M\right\|_{L^1_T(L^\infty)}\,\leq\,C\,,
\end{equation}
for a suitable constant $C>0$ independent of $M$.
%Recall that this notation means, in particular, that the family $\bigl(\nabla a_M\bigr)_M$ is uniformly bounded in $L^1_T(C^m_b)$.
Indeed, denoting by $\mathds{1}_A$ the characteristic function of a set $A\subset\R^d$ and by $B(x,R)$ the ball in $\R^d$ of center $x$ and radius $R>0$, we can compute
\begin{align*}
\left\|\nabla a_M\right\|_{L^\infty}\,&=\,\left\|\frac{1}{M}\,\nabla\chi\Big(\frac{x}{M} \Big)\,a\, +\, \chi\Big(\frac{x}{M} \Big)\,\nabla a\right\|_{L^\infty} \\
&\leq\,C\,\frac{1}{M}\,\left\|a\,\mathds{1}_{B(0,2M)}\right\|_{L^\infty}\,+\,\left\|\nabla a\right\|_{L^\infty}\;\leq\,C\,.
\end{align*}
The bounds for higher order derivatives follow the same lines, after noticing that, at each order of differentiation, we gain a factor $1/M$ in front of $a$.

At this point, for each fixed $M>0$, we can consider the truncated problem
\begin{align}
\label{eq:LM} 
\begin{cases}
	\partial_t \rho\, +\, \dive \left(a_M\,\rho\right)\, =\, g \\
	\rho_{|t=0}\, =\, \rho_0\,,
\end{cases}.
\end{align}
which possesses a unique weak solution $\rho_M\,\in\,C\bigl([0,T];H^m(\R^d)\bigr)$, by virtue of Theorem \ref{thm:ex-u_L}. Moreover, each $\rho_M$ satisfies
the energy estimate \eqref{est:en-estUnb}, up to replacing $a$ by $a_M$. Thus, we have 
\begin{align} \label{est:en-est_M}
\standardNorm{\rho_M(t)}_{H^m}\, \leq \,C
\left(\standardNorm{\rho_0}_{H^m}\, +\, \int_0^t \standardNorm{g(\tau)}_{H^m}\, d\tau \right)\;\exp\left(C\,\int_0^t\standardNorm{\nabla a_M(\tau)}_{C^{m}_b}\,d\tau \right)\,.
\end{align}
Thanks to property \eqref{ub:Da_M}, we deduce the uniform bounds
\begin{equation} \label{ub:rho^M}
\bigl(\rho_M\bigr)_M\,\subset\,L^\infty\bigl([0,T];H^m(\R^d)\bigr)\,.
\end{equation}

%\subsubsection{Convergence to a solution of the original equation} \label{sss:conv}

%Alfio: may be 'obtain' more appropriate than 'gather' 
As a consequence of \eqref{ub:rho^M}, we obtain the existence of a $\rho\,\in\,L^\infty_T(H^m)$ such that, up to the extraction of a subsequence, one has
$$
\rho_M\,\stackrel{*}{\rightharpoonup}\,\rho\qquad\qquad\mbox{ in }\qquad L^\infty_T(H^m)\,.
$$

Our next goal is to show that $\rho$ actually solves problem \eqref{eq:LiouvilleUnb} in the weak form, see equation \eqref{eq:LweakUnb}.
%Alfio: purpose
For this purpose, we need to pass to the limit, for $M\ra+\infty$, in the weak formulation
of equation \eqref{eq:LM}. Of course, it is enough to prove the convergence in the case of minimal regularity, namely for $m=0$. Thus, we restrict to this case in the next argument.

We start by recalling that $\rho_M$ is a weak solution to \eqref{eq:LM}: given any $\phi \in C^\infty_c\bigl(\RR^d \times [0,T[\,\bigr)$, we have
\begin{align}
-\int_0^T \SPI \rho_M\, \partial_t \phi\,dx\,dt \,-\, \int_0^T \SPI \rho_M\, a_M \cdot \nabla \phi\,dx\,dt \,=\, \int_0^T \SPI g\, \phi\,dx\,dt\, +\, \SPI \rho_0 \,\phi(0)\,dx\,.
\label{eq:LMweak}
\end{align}
The only term which presents some difficulties is the term $\rho_M\,a_M$, and thus we focus on it. We start by proving the following lemma.
\begin{lemma} \label{lem:LocConvergenceAM}
For all compact set $K \subset \RR^d$, it holds
\begin{align*}
\left\|a_M -a\right\|_{L^1_T(L^\infty(K))}\, \longrightarrow\, 0\qquad\qquad \mbox{ as }\quad M\,\ra\,+\infty\,.
\end{align*}
\end{lemma}

\begin{proof}
Let $K \subset \RR^d$ be a compact set, and let $R>0$ such that $K \subset B(0,R)$. The claim of the lemma follows then by noticing that, by definitions, for all $M \geq R+1$ one has
$a_M(t)\equiv a(t)$ over $K$, for almost every $t\in[0,T]$.
\end{proof}

Let now $K$ be the support of $\phi$, where $\phi$ is the test function appearing in \eqref{eq:LMweak}. Thanks to uniform bounds, to the strong convergence of $a_M$ to $a$ in
$L^1_T\bigl(L^\infty(K)\bigr)$ (given by Lemma \ref{lem:LocConvergenceAM}) and the weak-$*$ convergence of $\rho_M$ to $\rho$ in $L^\infty_T(L^2)$, we finally deduce that
$\bigl(\rho_M\,a_M\bigr)_M$ is uniformly bounded in $L^1_T\bigl(L^2(K)\bigr)$, and $\rho_M\,a_M\,\stackrel{*}{\rightharpoonup}\,\rho\,a$ in that space, in the limit when $M\ra+\infty$.

%Alfio: removed 'indeed' 

In the end, we have proved that the limit function $\rho$ is a weak solution to \eqref{eq:LiouvilleUnb}. Observe that, thanks to \eqref{est:en-est_M}, the uniform bounds \eqref{ub:Da_M}
and lower semicontinuity of the norm, we also deduce that $\rho$ verifies the energy estimate \eqref{est:en-estUnb}.

%\subsubsection{Uniqueness and time regularity of the solution} \label{sss:unique-time}
\paragraph{Time regularity and uniqueness.}
It remains to prove uniqueness of solutions and their time regularity. They are both consequences of the next proposition.
\begin{prop} \label{p:time-est}
Let $T>0$ and take $m\in\N$. Let $\rho\,\in\,L^\infty_T(H^m)$ be a weak solution to equation \eqref{eq:LiouvilleUnb} under hypotheses \eqref{hyp:data}.

Then $\rho\,\in\,C\bigl([0,T];H^m(\R^d)\bigr)$ and it verifies the energy estimate \eqref{est:en-estUnb}.
\end{prop}

We present the proof of the previous claim just in the minimal regularity case, namely for $m=0$. The general case follows by the same token.
To start with, let us state a classical lemma (see e.g. \cite{DiPernaLions1989} for details), whose proof is recalled in the appendix.
For this, we fix a function $s\,\in\,C^\infty_c(\R^d)$, with $s\equiv1$ for $|x|\leq1$ and $s\equiv0$ for $|x|\geq2$, $s$ radailly decreasing and such that $\int_{\R^d}s\,=\,1$.
For all $n\in\N$, we then define $s_n(x)\,:=\,n^d\,s(nx)$. We refer to the family $\big(s_n\big)_n$ as a family of standard mollifiers.
\begin{lemma} \label{l:commutator}
Let $\bigl(s_n\bigr)_n$ be a family of standard mollifiers, as constructed here above. For all $n\in\N$, define the operator $S_n$, acting on tempered distributions over $\R_+\times\R^d$,
by the formula
$$
S_n\rho\,:=\,s_n\,*_x\,\rho\,,
$$
where the symbol $*_x$ means that the convolution is taken only with respect to the space variable.
For given $\rho\,\in\,L^\infty_T(L^2)$ and $a\,\in\,L^1_T(C^1)$ such that $\nabla a\,\in\,L^1_T(C_b)$, we set, for all $n\in\N$ and $1\leq j\leq d$,
$$
r^j_n(\rho)\,:=\,\d_j\left(\big[a,S_n\big]\rho\right)\,.
$$

Then, for all $j$ fixed, we have $\big(r^j_n\big)_n\,\subset\,L^1_T(L^2)$; moreover, for $n\ra+\infty$, we have the strong convergence $r^j_n\,\longrightarrow\,0$ in $L^1_T(L^2)$.
\end{lemma}

Let us also recall the following standard notation. For $X$ a Banach space and $X^*$ its predual, we denote by $C_w\bigl([0,T];X\bigr)$ the set of measurable functions $f:[0,T]\ra X$
which are continuous with respect to the weak topology. Namely, for any $\phi\in X^*$, the function $t\,\mapsto\,\lan \phi,f(t)\ran_{X^*\times X}$ is continuous over $[0,T]$.

With this preparation, we are now ready to prove Proposition \ref{p:time-est}.
\begin{proof}[Proof of Proposition \ref{p:time-est}]
With the same notations as in Lemma \ref{l:commutator}, let us define $\rho_n\,:=\,S_n\rho$. Notice that $\big(\rho_n\big)_n\,\subset\,L^\infty_T(L^2)$. Moreover, $\rho_n$ satisfies the equation
\begin{equation} \label{eq:rho^n}
\d_t\rho_n\,+\,\div\big(a\,\rho_n\big)\,=\,g_n\,+\,r_n\,,\qquad\qquad\mbox{ with }\qquad \big(\rho_n\big)_{|t=0}\,=\,S_n\rho_0\,,
\end{equation}
where we have set $r_n\,:=\,\div\left(\big[a,S_n\big]\rho\right)$. Notice that one has $\left\|S_n\rho_0\right\|_{L^2}\,\leq\,C\,\left\|\rho_0\right\|_{L^2}$ and
$\left\|g_n\right\|_{L^1_T(L^2)}\,\leq\,C\,\|g\|_{L^1_T(L^2)}$.
Furthermore, in the limit $n\ra+\infty$, we have the strong convergence properties
$g_n\,\longrightarrow\,g$ in $L^1_T(L^2)$ and  $S_n\rho_0\,\longrightarrow\,\rho_0$ in $L^2$.
In addition, by Lemma \ref{l:commutator}, we know that $\left\|r_n\right\|_{L^1_T(L^2)}\,\leq\,C$ and $r_n\,\longrightarrow\,0$ in $L^1_T(L^2)$.

%Alfio: first of all removed ... 
Now, we remark that an easy inspection of \eqref{eq:rho^n} implies the property $\big(\d_t\rho_n\big)_n\,\subset\,L^1_T(H^{-1}_{\rm loc})$, which
in turn gives us the uniform embedding $\big(\rho_n\big)_n\,\subset\,C_T(H^{-1}_{\rm loc})$. From this latter property, combined with a density argument and the uniform
boundedness of $\big(\rho_n\big)_n$ in $L^\infty_T(L^2)$, we deduce that $\big(\rho_n\big)_n$ is uniformly bounded in $C_w\bigl([0,T];L^2(\R^d)\bigr)$.

Next, let us take the $L^2$ scalar product of equation \eqref{eq:rho^n} by $\rho_n$: by
standard computations we get
\begin{equation} \label{eq:L^2-norm}
\frac{1}{2}\,\frac{d}{dt}\left\|\rho_n\right\|^2_{L^2}\,+\,\frac{1}{2}\int\div a\,\left|\rho_n\right|^2\,dx\,=\,\int g_n\,\rho_n\,dx\,,
\end{equation}
which implies that, for all $n\in\N$, one has $\left\|\rho_n(t)\right\|_{L^2}\,\in\,C\big([0,T]\big)$. %, since $\big(\rho^n\big)_n\,\subset\,L^\infty_T(L^2)$.
Thanks to this property, together with the fact that $\rho_n\,\in\,C_w\bigl([0,T];L^2(\R^d)\bigr)$, after writing
$$
\left\|\rho_n(t+h)\,-\,\rho(t)\right\|_{L^2}^2\,=\,\left\|\rho_n(t+h)\right\|^2_{L^2}\,-\,2\,\lan\rho_n(t+h),\rho_n(t)\ran_{L^2\times L^2}\,+\,\left\|\rho_n(t)\right\|^2_{L^2}\,,
$$
%Alfio: immediately - spell
one immediately deduces that, for all $n\in\N$, $\rho_n$ belongs to $C_T(L^2)$.

On the other hand, by straightforward computations, from relation \eqref{eq:L^2-norm} we also infer the following inequality:
\begin{align}
\standardNorm{\rho_n(t)}_{L^2}\,&\leq\,C\,\exp\left(C\int_0^t \standardNorm{\div a(\tau)}_{L^\infty}\,d\tau \right)\,
\left(\standardNorm{S_n\rho_0}_{L^2}\,+\,\int_0^t \left(\standardNorm{g_n(\tau)}_{L^2}\,+\,\standardNorm{r_n(\tau)}_{L^2}\right)d\tau \right) \label{est:rho^n} \\
&\leq\,C\,\exp\left(C\int_0^t \standardNorm{\div a(\tau)}_{L^\infty}\,d\tau \right)\,
\left(\standardNorm{\rho_0}_{L^2}\,+\,\int_0^t\standardNorm{g(\tau)}_{L^2}\,d\tau \right) , \nonumber
\end{align}
for all $t \in [0,T]$, thanks also to the previous properties on $\big(S_n\rho_0\big)_n$, $\big(g_n\big)_n$ and $\big(r_n\big)_n$. In view of this energy estimate,
we deduce that $\big(\rho_n\big)_n$ is uniformly bounded in $C_T(L^2)$.

Next, we claim that $\left(\rho_n\right)_n$ is a Cauchy sequence in $C_T(L^2)$. For this, we take $m<n$ and consider the difference $\de_m^n\rho\,:=\,\rho_n-\rho_m$.
Then, $\de_m^n\rho$ fulfils
\begin{align*}
\d_t\de^n_m\rho\, +\, \div\big(a\,\de^n_m\rho\big)\,=\,\de^n_mg\,+\,\de^n_mr\,,\qquad\qquad \mbox{ with }\quad \de^n_m\rho_{|t=0}\,=\,\de^n_m\rho_0\,:=\,\rho^n_0\,-\,\rho_0^m\,,
\end{align*}
where we have defined also $\de^n_mg\,:=\,g_n-g_m$ and $\de^n_mr\,:=\,r_n-r_m$.
To this equation we can also apply the energy estimates, and obtain
\begin{align*}
\standardNorm{\de^n_m\rho}_{L^\infty_T(L^2)}\,\leq\,C\,\exp\left(C\,\standardNorm{\div a}_{L^1_T(L^\infty)} \right)\,
\left(\standardNorm{\de^n_m\rho_0}_{L^2}\,+\,\standardNorm{\de^n_mg}_{L^1_T(L^2)}\,+\,\standardNorm{\de^n_mr}_{L^1_T(L^2)}\right)\,.
\end{align*}
At this point, we can conclude thanks to the fact that $\big(S_n\rho_0\big)_n$, $\big(g_n\big)_n$ and $\big(r_n\big)_n$ are strongly convergent in the respective functional spaces,
and thus they are, in particular, Cauchy sequences. So, our claim is proved.

Further, we deduce that the limit $\rho$ of the sequence $\big(\rho_n\big)_n$ belongs to $C_T(L^2)$, and the convergence $\rho_n\longrightarrow\rho$ is strong in this space.
Finally, passing to the limit in the left-hand side of \eqref{est:rho^n} we discover that $\rho$ verifies the energy estimate \eqref{est:en-estUnb}.
\end{proof}

We conclude this part by remarking that stability, and then uniqueness, are easy consequences of Proposition \ref{p:time-est}.
\begin{prop} \label{p:L-stab}
Fix $T>0$ and $m\in\N$, and let $a$ be as in \eqref{hyp:data}. For $i=1,2$, take an initial datum $\rho_0^i\,\in\,H^m(\R^d)$ and an external force $g^i\,\in\,L^1\bigl([0,T];H^m(\R^d)\bigr)$,
and let $\rho^i\in L^\infty_T(H^m)$ be a corresponding solution to \eqref{eq:LiouvilleUnb} (whose existence is guaranteed by the previous arguments).

Then, after defining $\delta\rho_0\,:=\,\rho_0^1-\rho^2_0$, $\delta g\,:=\,g^1-g^2$ and $\delta\rho\,:=\,\rho^1-\rho^2$, the following estimate holds true for all $t\in[0,T]$, for some constant
$C$ independent of the data and the respective solutions:
$$
\left\|\delta\rho(t)\right\|_{H^m}\,\leq\,C\,\left(\standardNorm{\delta\rho_0}_{H^m}\, +\,\int_0^t \standardNorm{\delta g(\tau)}_{H^m}\, d\tau \right)\;
\exp\left(C\,\int_0^t \standardNorm{\nabla a(\tau)}_{C^{m}_b}\, d\tau \right)\,.
$$
\end{prop}

\begin{proof}
It is enough to remark that, by taking the difference of the equations satisfied by $\rho^1$ and $\rho^2$, one deduces that $\delta\rho\,\in\,L^\infty_T(L^2)$ is a weak solution to
the following equation:
\begin{align*}
%\label{eq:GronwallUniqueness}
\begin{cases}
\partial_t\delta\rho\, +\, \dive \bigl(a\,\delta\rho\bigr)\,=\,\delta g \\
\delta\rho_{|t=0}\, =\, \delta\rho_0\,.
\end{cases}.
\end{align*}

Then, Proposition \eqref{p:time-est} applies, and gives us the claimed estimates.
\end{proof} 

\subsubsection{The case of the transport equation} \label{sss:transport}
The characterization of ensemble controls with the optimality conditions given in Section \ref{sec-OptimalitySystem}, requires the solution of an adjoint Liouville problem,
which is given by a linear transport problem. In preparation of that discussion, and to complete the analysis of the present section, we consider the following transport problem
\begin{align}\label{eq:transportProblem} 
\begin{cases}
\partial_tq\, +\, a\cdot\nabla q\, +\, b\,q\, =\, g &\qquad \text{ in }\; [0,T]\times\R^d\\
q_{|t=0}\, =\, q_0 &\qquad \text{ on } \RR^d\,.
\end{cases}
\end{align}
We assume that the data $q_0$, $a$ and $g$ verify the assumptions in \eqref{hyp:data}, where $\rho_0$ is replaced by $q_0$.
Moreover, we assume that $b$ has the same regularity as $\dive a$: that is, $b\,\in\,L^1\bigl([0,T];C^m_b(\R^d)\bigr)$.

We point out that the weak formulation of \eqref{eq:transportProblem} now reads as follows: for all $\phi\, \in\, C_c^\infty\bigl(\RR^d \times [0, T[\,\bigr)$, one has the equality
\begin{align} 
-\int_0^T\!\!\!\int_{\RR^d}\rho\,\partial_t \phi\,-\int_0^T\!\!\! \int_{\RR^d}\rho\, a\cdot\nabla\phi\,-\int^T_0\!\!\!\int_{\R^d}\rho\,\div a\,\phi\,+\int^T_0\!\!\!\int_{\R^d}\rho\,b\,\phi\,=\,
\int_0^T\!\!\!\int_{\RR^d}g\,\phi\, + \int_{\RR^d} \rho_0\,\phi(0)\,.
\label{eq:Tr-weak}
\end{align}
 
For \eqref{eq:transportProblem}, we have the following well-posedness result, analogous to Theorem \ref{thm:existenceUnboudedA} for the Liouville equation.
\begin{theorem} \label{thm:ex-u_Tr}
Let us fix $T>0$ and $m\in\N$, and let the data $a$, $b$, $q_0$ and $g$ satisfy the assumptions stated above.

Then there exists a unique solution $q\,\in\,C\big([0,T];H^m(\R^d)\big)$ to equation \eqref{eq:transportProblem}.
Moreover, there exists a ``universal'' constant $C>0$, independent of $q_0$, $a$, $b$, $g$, $q$ and $T$, such that the following estimate holds true for any $t\in[0,T]$:
\begin{align} \label{est:en-est_Tr}
\standardNorm{q(t)}_{H^m}\, \leq \,C\,\left(\standardNorm{q_0}_{H^m}\, +\, \int_0^t \standardNorm{g(\tau)}_{H^m}\, d\tau \right)\;
\exp\left(C\,\int_0^t\left(\standardNorm{\nabla a(\tau)}_{C^{m}_b}\,+\,\standardNorm{b(\tau)}_{C^{m}_b}\right)\,d\tau \right)\,.
\end{align}
\end{theorem}

The proof is analogous to the one given for Theorem \ref{thm:ex-u_L}, so it is omitted here. In particular, the regularization procedure and the energy estimates work in exactly the same way.
The only point which deserves some attention is passing to the limit in the weak formulation \eqref{eq:Tr-weak} at step $n$ of the regularization procedure, especially in  the terms
$$
-\,\int^T_0\int_{\R^d}q^n\,\div a^n\,\phi\,dx\,dt\,+\,\int^T_0\int_{\R^d}q^n\,b^n\,\phi\,dx\,dt\,.
$$
Let us focus on the latter term only; the former can be treated by the same method. Of course, it is enough to treat the case when $m=0$.

%Alfio: first of all removed 
Now, we notice that the integral is in fact performed on the compact set $K\,:=\,{\rm supp}\,\phi$.
Therefore, by Proposition 4.21 and Theorem 4.22 of \cite{Brezis}, we deduce that $b^n\,\longrightarrow\,b$ in $L^1_T\bigl(L^\infty(K)\bigr)$ for $n\ra+\infty$.
On the other hand, thanks to uniform bounds, $q^n\,\stackrel{*}{\rightharpoonup}\,q$ in $L^\infty_T(L^2)$, for some $q$ belonging to that space; so, in particular the weak-$*$ convergence
holds true in $L^\infty_T\bigl(L^2(K)\bigr)$.
Putting these properties together, we deduce that $\bigl(q^n\,b^n\bigr)_n$ is uniformly bounded in $L^1_T\bigl(L^2(K)\bigr)$ and it weakly-$*$ converges to $q\,b$ in that space.

The previous argument shows that we can pass to the limit in the weak formulation of the approximated problems and gather that the limit point $q$ of the sequence $\bigl(q^n\bigr)_n$
solves \eqref{eq:Tr-weak}.

\subsection{Well-posedness theory in weighted spaces} \label{ss:weight}

In this section, we extend the previous theory to Sobolev spaces with weights.
This analysis is especially important for the investigation of the Liouville 
control-to-state map and of the Liouville ensemble optimal control problem, 
see the next sections.

\begin{remark} \label{r:tr-weight}
We limit ourselves to treat the case of the Liouville equation. However, 
the statements that follow can be proved
also for the transport problem, with slight modifications in the proofs.
\end{remark}

\subsubsection{Definition of weighted spaces} \label{sss:weight-sp}
For the analysis of the Liouville control-to-state map in Section \ref{s:control-map}, we need to prove weighted integrability of $\rho$,
due to the growth of the drift function. For this purpose, we introduce the following definition.
\begin{definition} \label{def:H^m_k}
Fix $(m,k)\in\N^2$. We define the space $H^m_k(\RR^d)$ in the following way:
$$
H^0_k(\R^d)\,=\,L^2_k(\R^d)\,:=\,\left\{f\in L^2(\R^d)\;\bigl|\quad|x|^k\,f\;\in\;L^2(\RR^d)\right\}\,,
$$
and, for $m\geq1$, we set
$$
H^m_k(\R^d)\,:=\,\left\{f\in H^m(\R^d)\cap H^{m-1}_k(\R^d)\;\big|\quad |x|^k\,D^\alpha f\;\in\;L^2(\RR^d)\quad\forall\;|\alpha|=m\right\}\,.
$$

The space $H^m_k$ is endowed with the following norm:
\begin{align*}
\standardNorm{f}_{H^m_k} := \sum_{|\alpha| \leq m} \left\|\big(1\,+\,|x|^k\big)\,D^\alpha f\right\|_{L^2}\,.
\end{align*}
\end{definition}

Sometimes, given $m\in\N$, we will use the notation
$$
\left\|\nabla^mf\right\|_{L^2}\,=\,\sum_{|\alpha|=m}\left\|D^\alpha f\right\|_{L^2}\,
$$
and analogous writing for weighted norms.

Notice that, for all fixed $m$ and $k$ in $\N$, one has the embedding $H^m_k\,\subset\,H^m$. Of course, $H^m\,=\,H^m_0$ for all $m\geq0$.
Furthermore, since we want to avoid too singular behaviours
close to $0$, we will often focus on the special case (which will be enough for our scopes)
$$
m\,\leq\,k\,.
$$
Then, we have a simple characterization of the spaces $H^m_k$, which will be useful especially in Section \ref{s:control-map},
when studying the control-to-state map related to our optimal control problem. 
\begin{prop} \label{p:H^m_k}
\begin{enumerate}[(i)]
 \item Given $k\in\N$, one has $f\,\in\,L^2_k$ if and only if $(1+|x|^k)\,f\,\in\,L^2$.
\item For $k\in\N\setminus\{0\}$ and $1\leq m\leq k$, let $f\,\in\,H^m\cap H^{m-1}_k$. Then $f\,\in\,H^m_k$ if and only if $|x|^k\,f\,\in\,H^m$. \\
In particular, a tempered distribution $f$ belongs to $H^1_1$ if and only if both $f$ and $|x|\,f$ belong to $H^1$; it belongs to $H^2_2$ if and only if both $f$ and
$|x|^2\,f$ belong to $H^2$ and $\nabla f$ belongs to $L^2_2$.
%%% Analogously, suppose $k\geq1$. For any $1\leq m\leq k$, we have that a function $f\,\in\,H^m_k$ if and only if $(1+|x|^k)\,f\,\in\,H^m$.
\end{enumerate}
\end{prop}

Before proving this statement, let us state a preliminary result, whose proof is postponed to the Appendix.
\begin{lemma} \label{l:H^m_k}
Let $(m,k)\in\N^2$, with $m\leq k$.
If $f\,\in\,H^m_k$, then $(1+|x|^k)\,f\,\in\,H^{m}$.
\end{lemma}

Thanks to Lemma \ref{l:H^m_k}, we can prove Proposition \ref{p:H^m_k}.
\begin{proof}[Proof of Proposition \ref{p:H^m_k}]
Assertion (i) is elementary. So, let us focus on the proof of (ii).

Suppose that $f\in H^m\cap H^{m-1}_k$. Then, by Lemma \ref{l:H^m_k} above, we have that $|x|^k\,f\,\in\,H^{m-1}_k$.
At this point, for $|\alpha|=m$, we write, using again Leibniz rule,
$$
D^\alpha\left(|x|^k\,f\right)\,=\,|x|^k\,D^\alpha f\,+\,\sum_\beta D^\beta|x|^k\,D^{\alpha-\beta}f\,,
$$
where the sum is performed for all $\beta\leq\alpha$ such that $|\beta|\geq1$. By the previous arguments, and the fact that $m\leq k$, we have that all the terms in the sum belong to $L^2$.
Then, the term on the left-hand side belongs to $L^2$ if and only if the first term on the right-hand side does.

The last sentences follow by straightforward computations. First of all, we have that
$$
\d_j\big(|x|\,f\big)\,=\,\d_j|x|\,f\,+\,|x|\,\d_jf\,,
$$
for all $1\leq j\leq d$. Furthermore, we also have
$$
\nabla^2\big(|x|^2\,f\big)\,\sim\,\nabla \big(|x|\,f\,+\,|x|^2\,\nabla f\big)\,\sim\,\nabla|x|\,f\,+\,\big(|x|+|x|^2\big)\,\nabla f\,+\,|x|^2\,\nabla^2f\,.
$$
The equivalence between the two assertions is then apparent. Indeed, arguing as in the beginning of the proof to Lemma \ref{l:H^m_k}, we gather that, if $f\,\in\,H^2_2$, then
$|x|^j\,D^\alpha f\,\in\,L^2$ for all $0\leq j\leq 2$ and $|\alpha|=0, 1$. Hence, all the terms in the right-hand side belong to $L^2$, and then so does the one on the left-hand side.
On the contrary, if both $f$ and $|x|^2\,f$ belong to $H^2$ and $\nabla f$ belongs to $L^2_2$, then $f\,\in\,H^2\cap H^1_2$; finally, by the previous equality, we also discover
that $|x|^2\,\nabla^2f$ belongs to $L^2$, completing the proof of the reverse implication.

The proof of the proposition is hence completed. 
\end{proof}

\subsubsection{The Liouville equation in weighted spaces} \label{sss:Liouville-w}
After the above preliminaries, we are ready to state the main result of this section, which show well-posedness of the Liouville equation in $H^m_k$ spaces.
\begin{theorem} \label{th:weight}
Let $T>0$ and $(m,k)\in\N^2$ fixed, and let $a$ be a vector field satisfying hypotheses \eqref{hyp:data}. Moreover, assume that $\rho_0\in H^m_k(\R^d)$ and $g\in L^1\big([0,T];H^m_k(\R^d)\big)$.

Then there exists a unique solution $\rho\,\in\,C\bigl([0,T];H^{m}_k(\RR^d)\bigr)$ to problem \eqref{eq:LiouvilleUnb}.
Moreover, there exists a ``universal'' constant $C>0$, independent of $\rho_0$, $a$, $g$, $\rho$ and $T$, such that the following estimate holds true for any $t\in[0,T]$:
\begin{align} \label{est:weight}
\standardNorm{\rho(t)}_{H^m_k}\, \leq\,C\,\exp\left(C\,\int_0^t\left\|\nabla a(\tau)\right\|_{C^m_b}\,d\tau \right)\,
\left(\standardNorm{\rho_0}_{H^m_k}\, +\, \int_0^t \standardNorm{g(\tau)}_{H^m_k}\, d\tau \right)\,.
\end{align}
\end{theorem}

Before proving this statement in its full generality, let us consider its version for simpler cases, which will be needed in the proof of the general case. 
Moreover, their precise form is important, in view of their application in Section \ref{s:control-map}.

We start with the case $m=0$.

\begin{lemma} \label{l:weight-m=0}
Assume that the hypotheses of Theorem \ref{th:weight} hold true with $m=0$.

Then there exists a unique solution $\rho\,\in\,C\bigl([0,T];L^2_k(\RR^d)\bigr)$ to problem \eqref{eq:LiouvilleUnb}. Moreover, there  exists a ``universal'' constant $C>0$
such that the following estimate holds true for any $t\in[0,T]$:
\begin{align*} %\label{est:weight}
\standardNorm{\rho(t)}_{L^2_k}\, \leq\,C\,\exp\left(C\,\int_0^t\left\|\nabla a(\tau)\right\|_{L^\infty}\,d\tau \right)\,
\left(\standardNorm{\rho_0}_{L^2_k}\, +\, \int_0^t \standardNorm{g(\tau)}_{L^2_k}\, d\tau \right)\,.
\end{align*}
\end{lemma} 

\begin{proof}[Proof of Lemma \ref{l:weight-m=0}]
Most of the claims of the Lemma follow from Theorem \ref{thm:existenceUnboudedA}. We have just to prove propagation of higher integrability (i.e. $k\geq1$) of the initial datum and external force.
Omitting a standard regularization procedure for the sake of brevity, we will perform energy estimates directly on equation \eqref{eq:LiouvilleUnb}.

%Alfio: first of all ..... 
First of all, for completeness, we consider the case $k=0$. If we take the $L^2$ scalar product of equation \eqref{eq:LiouvilleUnb} by $\rho$, by
standard computations we get
$$
\frac{1}{2}\,\frac{d}{dt}\left\|\rho\right\|^2_{L^2}\,+\,\frac{1}{2}\int\div a\,|\rho|^2\,dx\,=\,\int g\,\rho\,dx\,.
$$
From this relation, we easily get
\begin{equation} \label{est:m=0_0}
\frac{d}{dt}\left\|\rho\right\|_{L^2}\,\leq\,\|\div a\|_{L^\infty}\,\left\|\rho\right\|_{L^2}\,+\,\left\|g\right\|_{L^2}\,.
\end{equation}

Next, let us multiply equation \eqref{eq:LiouvilleUnb} by $|x|^k$: we get that $\rho_k\,:=\,|x|^k\,\rho$ satisfies
$$
\d_t\rho_k\,+\,\div\big(a\,\rho_k\big)\,=\,|x|^k\,g\,+\,\rho\,a\cdot\nabla|x|^k\,.
$$
Taking the $L^2$ scalar product of this equation by $\rho_k$ and repeating the same computations as above, we find
\begin{equation} \label{est:m=0_dt}
\frac{d}{dt}\left\|\rho_k\right\|_{L^2}\,\leq\,\|\div a\|_{L^\infty}\,\left\|\rho_k\right\|_{L^2}\,+\,\left\||x|^k\,g\right\|_{L^2}\,+\,\left\|\rho\,a\cdot\nabla|x|^k\right\|_{L^2}\,.
\end{equation}

We need to control the last term on the right-hand side of the previous estimate. For this, we use the fact that $\nabla|x|^k\,\sim\,|x|^{k-1}$ for all $k\geq1$, and Remark \ref{r:unb-drift}, to obtain
$$
\left\|\rho\,a\cdot\nabla|x|^k\right\|_{L^2}\,\leq\,C\,\|\nabla a\|_{L^\infty}\,\left\|\big(1+|x|^k\big)\,\rho\right\|_{L^2}\,.
$$
Inserting this bound into \eqref{est:m=0_dt} and summing up the resulting expression to \eqref{est:m=0_0}, we have 
\begin{equation} \label{est:m=0-tot}
\frac{d}{dt}\left\|\big(1+|x|^k\big)\,\rho\right\|_{L^2}\,\leq\,C\,\|\nabla a\|_{L^\infty}\,\left\|\big(1+|x|^k\big)\,\rho\right\|_{L^2}\,+\,\left\|\big(1+|x|^k\big)\,g\right\|_{L^2}\,.
\end{equation}
Hence, an application of Gr\"onwall's lemma gives the desired estimate.
\end{proof}

Next, we present results for  $m=1$.
\begin{lemma} \label{l:weight-m=1}
Assume that the hypotheses of Theorem \ref{th:weight} hold true with $m=1$.

Then there exists a unique solution $\rho\,\in\,C\bigl([0,T];H^1_k(\RR^d)\bigr)$ to problem \eqref{eq:LiouvilleUnb}. Moreover, there  exists a ``universal'' constant $C>0$
such that the following estimate holds true for any $t\in[0,T]$:
\begin{align*} %\label{est:weight}
\standardNorm{\rho(t)}_{H^1_k}\, \leq\,C\,\exp\left(C\,\int_0^t\left\|\nabla a(\tau)\right\|_{C^1_b}\,d\tau \right)\,
\left(\standardNorm{\rho_0}_{H^1_k}\, +\, \int_0^t \standardNorm{g(\tau)}_{H^1_k}\, d\tau \right)\,.
\end{align*}
\end{lemma}

\begin{proof}[Proof of Lemma \ref{l:weight-m=1}]
Once again, it is enough to focus on the proof of the energy estimates. We start by differentiating equation \eqref{eq:LiouvilleUnb} with respect to $x^j$, for some $1\leq j\leq d$: we get
$$
\d_t\d_j\rho\,+\,\div\big(a\,\d_j\rho\big)\,=\,\d_jg\,-\,\d_j\div a\;\rho\,-\,\d_ja\cdot\nabla\rho\,.
$$
Applying estimate \eqref{est:m=0-tot} to this equation gives
\begin{align*}
\frac{d}{dt}\left\|\big(1+|x|^k\big)\,\d_j\rho\right\|_{L^2}\,&\leq\,C\,\|\nabla a\|_{L^\infty}\,\left\|\big(1+|x|^k\big)\,\d_j\rho\right\|_{L^2}\,+\,\left\|\big(1+|x|^k\big)\,\d_jg\right\|_{L^2}\,+ \\
&\qquad\qquad+\,\left\|\big(1+|x|^k\big)\,\d_j\div a\;\rho\right\|_{L^2}\,+\,\left\|\big(1+|x|^k\big)\,\d_ja\cdot\nabla\rho\right\|_{L^2}\,,
\end{align*}
from which we obtain, for another constant $C>0$, the following bound:
\begin{align}
\hspace{-0.5cm} \frac{d}{dt}\left\|\big(1+|x|^k\big)\,\nabla\rho\right\|_{L^2}\,&\leq\,C\,\|\nabla a\|_{L^\infty}\,\left\|\big(1+|x|^k\big)\,\nabla\rho\right\|_{L^2}\,+\,
 \left\|\big(1+|x|^k\big)\,\nabla g\right\|_{L^2}\,+  \label{est:m=1-part} \\
&\qquad\qquad\qquad\qquad\qquad\qquad\qquad\quad+\,\left\|\nabla^2a\right\|_{L^\infty}\,\left\|\big(1+|x|^k\big)\,\rho\right\|_{L^2}\,. \nonumber
\end{align}

We can now sum up \eqref{est:m=0-tot} and \eqref{est:m=1-part} to get
\begin{equation} \label{est:m=1_tot}
\frac{d}{dt}\left\|\rho\right\|_{H^1_k}\,\leq\,C\,\left\|\nabla a\right\|_{C^1_b}\,\left\|\rho\right\|_{H^1_k}\,+\,\left\|g\right\|_{H^1_k}\,,
\end{equation}
and Gr\"onwall's lemma allows us to get the result.
\end{proof}

Now, we can address the proof of the general case, namely of Theorem \ref{th:weight}.
\begin{proof}[Proof of Theorem \ref{th:weight}]
We argue by induction on the order of derivatives, i.e. on $m$, the cases $m=0$ and $m=1$ being given by Lemma \ref{l:weight-m=0} and Lemma \ref{l:weight-m=1}, 
respectively.

Let $m\geq2$, and let us assume that, for any $0\leq\ell\leq m-1$, the following inequality holds true
\begin{align}
\hspace{-0.5cm} \frac{d}{dt}\left\|\big(1+|x|^k\big)\,\nabla^\ell\rho\right\|_{L^2}\,&\leq\,C\,\|\nabla a\|_{L^\infty}\,\left\|\big(1+|x|^k\big)\,\nabla^\ell\rho\right\|_{L^2}\,+\,
 \left\|\big(1+|x|^k\big)\,\nabla^\ell g\right\|_{L^2}\,+  \label{est:m-part} \\
&\quad\qquad\qquad\qquad\qquad +\,\sum_{0\leq p\leq \ell-1}\left\|\nabla^{p+1}a\right\|_{L^\infty}\,\left\|\big(1+|x|^k\big)\,\nabla^p\rho\right\|_{L^2}\,. \nonumber
\end{align}
Our goal is to prove an analogous estimate also for $\left\|\big(1+|x|^k\big)\,\nabla^m\rho\right\|_{L^2}$.

For this purpose, let us take an $\alpha\in\N^d$ such that $|\alpha|=m$, and let us apply it to \eqref{eq:LiouvilleUnb}. Thus, we deduce the following equality
\begin{equation} \label{eq:D^alpha}
\d_tD^\alpha\rho\,+\,\div\big(a\,D^\alpha\rho\big)\,=\,D^\alpha g\,-\,\sum_{0<\beta\leq\alpha}D^\beta\div a\;D^{\alpha-\beta}\rho\,-\,\sum_{0<\beta\leq\alpha}D^\beta a\cdot\nabla D^{\alpha-\beta}\rho\,,
\end{equation}
where the notation $0<\beta$ means that $\beta\in\N^d$  has at least one non-zero component.

Following the computations of Lemma \ref{l:weight-m=1}, we need to estimate the $L^2_k$ norm of the last two terms in the right-hand side of the previous equation.
First of all, we have
$$
\left\|\big(1+|x|^k\big)\,D^\beta\div a\;D^{\alpha-\beta}\rho\right\|_{L^2}\,\leq\,\left\|\nabla^{|\beta|+1}a\right\|_{L^\infty}\,\left\|\big(1+|x|^k\big)\,D^{\alpha-\beta}\rho\right\|_{L^2}\,.
$$
Notice that, since $\beta>0$, the terms $D^{\alpha-\beta}\rho$ are lower order. The same can be said of the terms
$$
\left\|\big(1+|x|^k\big)\,D^\beta a\cdot\nabla D^{\alpha-\beta}\rho\right\|_{L^2}\,\leq\,\left\|\nabla^{|\beta|}a\right\|_{L^\infty}\,\left\|\big(1+|x|^k\big)\,\nabla D^{\alpha-\beta}\rho\right\|_{L^2}\,,
$$
whenever $|\beta|\geq2$; on the contrary, when $|\beta|=1$, the terms $\nabla D^{\alpha-\beta}\rho$ contain exactly $m$ derivatives.

Therefore, applying estimate \eqref{est:m=0-tot} to equation \eqref{eq:D^alpha}, and using the previous controls, we infer
\begin{align} \label{est:m}
\hspace{-0.7cm}
\frac{d}{dt}\left\|\big(1+|x|^k\big)\,\nabla^m\rho\right\|_{L^2}\,&\leq\,C\,\|\nabla a\|_{L^\infty}\,\left\|\big(1+|x|^k\big)\,\nabla^m\rho\right\|_{L^2}\,+\,
\left\|\big(1+|x|^k\big)\,\nabla^mg\right\|_{L^2}\,+ \\
&\quad\qquad\qquad\qquad\qquad +\,\sum_{0<\beta\leq\alpha}\left\|\nabla^{|\beta|+1}a\right\|_{L^\infty}\,\left\|\big(1+|x|^k\big)\,D^{\alpha-\beta}\rho\right\|_{L^2} \nonumber \\
&\leq\,C\,\|\nabla a\|_{L^\infty}\,\left\|\big(1+|x|^k\big)\,\nabla^m\rho\right\|_{L^2}\,+\, \left\|\big(1+|x|^k\big)\,\nabla^mg\right\|_{L^2}\,+ \nonumber \\
&\quad\qquad\qquad\qquad\qquad +\,\sum_{0\leq\ell\leq m-1}\left\|\nabla^{\ell+1}a\right\|_{L^\infty}\,\left\|\big(1+|x|^k\big)\,\nabla^\ell\rho\right\|_{L^2}\,, \nonumber
\end{align}
which proves formula \eqref{est:m-part} at the level $m$. Therefore that formula is true for any $m\in\N$, by induction.

Now, summing up inequality \eqref{est:m=1-part} for $\ell=0$ to $m$, and by the definition of $H^m_k$ norms, we get, for some constant also depending on $m$, the following bound: 
$$
\frac{d}{dt}\left\|\rho\right\|_{H^m_k}\,\leq\,C\,\left\|\nabla a\right\|_{C^m_b}\,\left\|\rho\right\|_{H^m_k}\,+\,\left\|g\right\|_{H^m_k}\,,
$$
which immediately implies the claimed estimate. Theorem \ref{th:weight} is now proved.
\end{proof} 

\section{The Liouville control-to-state map} \label{s:control-map}

In this section, we define the Liouville control-to-state map and investigate its continuity and differentiability properties.
For reasons which will appear clear in the following analysis, we need to resort to weighted spaces $H^m_k$, as introduced in Section \ref{ss:weight}. 

We start by making an important remark.
\begin{remark} \label{r:data}
Throughout this section, the data of the Liouville equation has to be thought as fixed. Specifically,
for $m\geq0$ and $k\geq0$, we take an initial datum $\rho_0\,\in\,H^m_k$, a source term 
$g\,\in\,L^1_T(H^m_k)$, and a drift function $a_0\,\in\,L^1_T(C^{m+1})$, with $\nabla a_0\,\in\,L^1_T(C^m_b)$.

We are then interested in the dependence of the solution $\rho$ to the Liouville
equation \eqref{eq:LiouvilleUnb}, with drift $a$ given by \eqref{controlmechanism},
on the control state $u\in U_{ad}$, where $U_{ad}$ has been defined in \eqref{setUad}.
\end{remark}

\begin{comment}
Let us start by recalling that,
with the energy estimate, we can conclude the \enquote{conservation} of the solution with respect to $\rho_0$ and $g$. Let us define $C=C(u_{\max},T):=\exp(Tu_{\max})$ with $u_{\max}:=\max\{|u_a|,|u_b| \}$ then follows immediately:
\begin{corollary}
\label{cor:conservation}
Let $\rho$ be the solution of \eqref{eq:LiouvilleUnb} given by Theorem \ref{thm:ex-u_L}  and let $a,\rho_0$ and $g$ satisfy hypothesis \eqref{hyp:data}. Then there exists a constant $C=C(u_{\max},T)>0$, such that 
	$$\standardNorm{\rho(t)}_{H^m} \leq C \standardNorm{\rho_0}_{H^m}+C\standardNorm{g}_{L^1H^m}~~~~~~\forall t \in [0,T].$$
\end{corollary}
\end{comment}

\subsection{Definition and continuity properties} \label{ss:G-def}
We remark that the statements of Theorems \ref{thm:existenceUnboudedA} and \ref{th:weight} cover the case of the Liouville equation with the controlled drift function given by \eqref{controlmechanism}, where 
$u \in U_{ad}$. In particular, the next proposition-definition immediately follows.
\begin{prop} \label{p:G}
Fixed data $\rho_0$, $g$ and $a_0$ as in Remark \ref{r:data}, let us consider drift functions $a$ of the form \eqref{controlmechanism}, with $u\in U_{ad}$.
Introduce the \emph{Liouville control-to-state map} $G$, defined by
$$
G:\, U_{ad}\, \longrightarrow\, L^\infty\big([0,T];L^2(\R^d)\big)\,,  \qquad  u\, \mapsto\, \rho := G(u)\,,
$$
where $\rho$ is the unique solution to the Liouville equation with the given data.

Then $G$ is well-defined.
\end{prop}

Let us make an important comment about the previous definition.
\begin{remark} \label{r:G-loss}
Notice that the theory developed in Sections \ref{ss:classical} and \ref{ss:weight} entails that the solution $\rho$ actually belongs to $C_T(H^m_k)$.
However, for reasons that will appear clear in what follows (namely, a \emph{loss of regularity}, both in $m$ and $k$, when proving Fr\'echet differentiability of $G$),
it is convenient to look at $G$ as a map with values in the space with the weakest topology. Notice that, for any $(m,k)\in\N^2$, the space $H^m_k$ is embedded in $L^2$.

Finally, we consider $L^\infty$ regularity with respect to time, because it will be convenient also to look at weak continuity properties of $G$, see Proposition \ref{p:ww-G} below.
\end{remark}

Next, we study some properties of the map $G$ that are relevant for the analysis of ensemble optimal control problems.
We start by establishing that $G$ is weak-weak continuous from $U_{ad}$ into $L^\infty_T(L^2)$. Notice that we do not need any restriction on $m$ and $k$ (and so, on the initial data) in this case.
\begin{prop} \label{p:ww-G}
Take $m\geq0$ and $k\geq0$ and initial data $\rho_0\,\in\,H^m_k$, $g\,\in\,L^1_T(H^m_k)$ and  $a_0\,\in\,L^1_T(C^{m+1})$ such that $\nabla a_0\,\in\,L^1_T(C^m_b)$.
Let $u\in U_{ad}$ and $\big(u^l\bigr)_l\,\subset\,U_{ad}$ be a sequence of controls, and assume that $u^l\,\stackrel{*}{\rightharpoonup}\,u$ in $\LL^\infty_T$.

Then $G(u^l)\,\stackrel{*}{\rightharpoonup}\,G(u)$ in the weak-$*$ topology of $L^\infty_T(L^2)$.
\end{prop}

\begin{proof}
Of course, it is enough to prove the previous proposition in the case of minimal regularity and integrability, namely for $m=k=0$.

By definition of the set $U_{ad}$, we infer that $(u^l)_l$ is uniformly bounded in $\LL^\infty_T$. On the other hand, by hypotheses and Theorem \ref{thm:existenceUnboudedA},
for all $l\in\N$ there exists a unique $\rho^l\,:=\,G(u^l)\,\in\,C_T(L^2)$ which solves the Liouville equation \eqref{eq:LiouvilleUnb}.
In addition, by inequality \eqref{est:en-estUnb}, we deduce that $\big(\rho^l\big)_l$ is uniformly bounded in $C_T(L^2)$; then there exists $\rho\in L^\infty_T(L^2)$
such that, up to extraction of a subsequence, $\rho^l\,\stackrel{*}{\rightharpoonup}\,\rho$ in $L^\infty_T(L^2)$.

Now, our goal is to prove that $\rho$ is a weak solution to the Liouville equation
\begin{equation} \label{eq:limit_rho}
\d_t\rho\,+\,\div\big(a(t,x;u)\,\rho\big)\,=\,g\,,\qquad\qquad\mbox{ with }\quad \rho_{|t=0}\,=\,\rho_0\,.
\end{equation}
Indeed, if this is the case, by uniqueness we get $\rho\,=\,G(u)$ and that the whole sequence $\big(\rho^l\big)_l$ converges, achieving in this way the proof of the proposition. 

In order to prove our claim, we need to pass to the limit in the weak formulation of the Liouville equation for $\rho^l$, when $l\ra+\infty$. Recalling also our special choice
\eqref{controlmechanism}, it is easy to see that the only term which presents some difficulty is the non-linear term
\begin{equation} \label{eq:integral}
\int^T_0\int_{\R^d}\,\rho^l\,\left(u^l_1(t)\,+\,x\circ u^l_2(t)\right)\cdot\nabla\phi\,dx\,dt\,,\qquad\qquad \mbox{ for any fixed }\quad \phi\,\in\,C^\infty_c\bigl([0,T[\,\times\R^d\bigr)\,.
\end{equation}
Therefore, let us focus on the convergence of this integral. First of all, by inspection of the equation $\d_t\rho^l\,=\,-\div\big(a(x,t;u^l)\,\rho^l\big)\,+\,g$, we discover that
$\big(\d_t\rho^l\big)_l\,\subset\,L^1_T(H^{-1}_{\rm loc})$, which implies that $\big(\rho^l\big)_l\,\subset\,W^{1,1}_T(H^{-1}_{\rm loc})$. Then, by the Rellich-Kondrachov theorem and
Cantor's diagonal procedure, we discover that, up to an extraction of a subsequence that we do not relabel, $\big(\rho^l\big)_l$ is compact, and then strongly convergent, in $L^1_T(H^{-2}_{\rm loc})$.
Interpolating this compactness result with the uniform boundedness in $L^\infty_T(L^2_{\rm loc})$, we discover that $\rho^l\,\longrightarrow\,\rho$ in $L^1_T(H^{-s}_{\rm loc})$, for
any $s>0$.

In view of the uniform boundedness of $\big(u^l_1\big)_l$ and $\big(u^l_2\big)_l$ in $L^\infty_T$, the previous property is enough to pass to the limit in the integral \eqref{eq:integral}, and prove that it converges to
$$
\int^T_0\int_{\R^d}\,\rho\,\left(u_1(t)\,+\,x\circ u_2(t)\right)\cdot\nabla\phi\,dx\,dt\,,
$$
for all given $\phi\,\in\,C^\infty_c\bigl([0,T[\,\times\R^d\bigr)$. Thus, we get that \eqref{eq:limit_rho} is satisfied, and then we can conclude the proof as already mentioned above.
\end{proof}

For the analysis of our optimal control problem, see Section \ref{s:ocp} below, we need stronger regularity properties for $G$.
We start by showing Lipschitz continuity, which will be the basis to prove G\^ateaux differentiability of $G$, in the next paragraph. The key here is to perform careful estimates
in order to identify the right topology: the reason is that, due to hyperbolicity of the Liouville equation, stability estimates involve a loss of regularity.

\begin{lemma} \label{l:Lip}
Let the data $\rho_0$, $g$ and $a_0$ be fixed as in Remark \ref{r:data} above, with $m\geq1$ and $k\geq1$. Let $u$ and $v$ be in $U_{ad}$, and denote by $G(u)$ and $G(v)$ 
the corresponding $C_T(H^m_k)$ solutions to \eqref{eq:LiouvilleUnb}, with drift $a$ given by \eqref{controlmechanism}. Set $\de G\,:=\,G(u)-G(v)$.

Then there exists a ``universal'' constant $C>0$ such that, for all $1\leq \ell\leq k$, if we set 
\begin{equation} \label{def:K_0}
K^{(\ell)}_0\,:=\,C\,\exp\Big(C\left(\left\|\nabla a_0\right\|_{L^1_T(C^1_b)}\,+\,\|u\|_{\LL^1_T}\,+\,\|v\|_{\LL^1_T}\right)\Big)\times
\left(\left\|\rho_0\right\|_{H^1_\ell}\,+\,\left\|g\right\|_{L^1_T(H^1_\ell)}\right)\,,
\end{equation}
then, for all $t\in[0,T]$, one has
\begin{align*}
\left\|\de G(t)\right\|_{L^2_{\ell-1}}\, %+\,\standardNorm{\de G(t)}_{H^1_1}\,
\leq\,K^{(\ell)}_0\,\int^t_0|u(\tau)-v(\tau)|\,d\tau\,.
\end{align*}
If moreover $m\geq2$ and we set
\begin{equation} \label{def:K_1}
K^{(\ell)}_1\,:=\,C\,\exp\Big(C\left(\left\|\nabla a_0\right\|_{L^1_T(C^2_b)}\,+\,\|u\|_{\LL^1_T}\,+\,\|v\|_{\LL^1_T}\right)\Big)\times
\left(\left\|\rho_0\right\|_{H^2_\ell}\,+\,\left\|g\right\|_{L^1_T(H^2_\ell)}\right)\,,
\end{equation}
we also have
\begin{align*}
\standardNorm{\de G(t)}_{H^1_{\ell-1}}\,\leq\,K^{(\ell)}_1\,\int^t_0|u(\tau)-v(\tau)|\,d\tau\,.
\end{align*}
\end{lemma}

\begin{proof}
%Alfio: first of all removed
By linearity of the Liouville equation, we find that $\de G$ satisfies
\begin{equation} \label{eq:dG}
\d_t\de G\,+\,\div\big(a(t,x;u)\,\de G\big)\,=\,-\,\div\big(\oline{a}(t,x;u-v)\,G(v)\big)\,,\qquad\qquad \de G_{|t=0}\,=\,0\,,
\end{equation}
where we have set
\begin{equation} \label{def:a(u)-a(v)}
\oline{a}(t,x;u-v)\,:=\,a(t,x;u)-a(t,x;v)\,=\,(u_1-v_1)+x\circ(u_2-v_2)\,.
\end{equation}

Applying $L^2_{\ell-1}$ estimates of Theorem \ref{thm:existenceUnboudedA} to equation \eqref{eq:dG}, we immediately get
$$
\left\|\de G(t)\right\|_{L^2_{\ell-1}}\,\leq\,C\,\exp\left(C\int^t_0\left\|\nabla a(\tau,x;u)\right\|_{L^\infty}\,d\tau\right)\,
\int^t_0\left\|\div\big(\oline{a}(\tau,x;u-v)\,G(v)\big)\right\|_{L^2_{\ell-1}}\,d\tau\,.
$$
By explicit computations and using the Leibniz rule, we deduce that
\begin{align}
&\left\|\div\big(\oline{a}(\tau,x;u-v)\,G(v)\big)\right\|_{L^2_{\ell-1}}\,\leq\,|u(\tau)-v(\tau)|\,\left(\left\|G(v)\right\|_{L^2_{\ell-1}}\,+\,\left\|\nabla G(v)\right\|_{L^2_\ell}\right)
\label{est:dG_L^2_part} \\
&\qquad\qquad
\leq\,C\,|u(\tau)-v(\tau)|\,\exp\left(C\int^\tau_0\left\|\nabla a(s,x;v)\right\|_{C^1_b}\,ds\right)\,\left(\|\rho_0\|_{H^1_\ell}+\int^\tau_0\|g(s)\|_{H^1_\ell}\,ds\right)\,, \nonumber
\end{align}
where the second inequality holds true in view of the bound
$\left\|G(v)\right\|_{L^2_{\ell-1}}+\left\|\nabla G(v)\right\|_{L^2_\ell}\,\leq\,\left\|G(v)\right\|_{H^1_\ell}$ and
Lemma \ref{l:weight-m=1}. This estimate completes the proof of the first inequality, for $L^2$-type norms of $\de G$.

Now, we focus on $H^1_{\ell-1}$ bounds for $\de G$. Thanks to Lemma \ref{l:weight-m=1}, we have
\begin{equation} \label{est:dG_H^1_1}
\left\|\de G(t)\right\|_{H^1_{\ell-1}}\,\leq\,C\,\exp\left(C\int^t_0\left\|\nabla a(\tau,x;u)\right\|_{C^1_b}\,d\tau\right)\,
\int^t_0\left\|\div\big(\oline{a}(\tau,x;u-v)\,G(v)\big)\right\|_{H^1_{\ell-1}}\,d\tau\,.
\end{equation}
By definition, we have that $\|f\|_{H^1_{\ell-1}}\,=\,\|f\|_{L^2_{\ell-1}}+\|\nabla f\|_{L^2_{\ell-1}}$. Then, we start with the bound
\begin{align}
\left\|\div\big(\oline a(\tau,x;u-v)\,G(v)\big)\right\|_{L^2_{\ell-1}}\,&=\,\left\|\div\oline a\;G(v)\right\|_{L^2_{\ell-1}}\,+\, 
\left\|\oline a\cdot\nabla G(v)\right\|_{L^2_{\ell-1}} \label{est:dG_p1} \\
&\leq\,C\,|u(\tau)-v(\tau)|\,\left(\left\|G(v)\right\|_{L^2_{\ell-1}}\,+\,\left\|\nabla G(v)\right\|_{L^2_{\ell}}\right) \nonumber \\
&\leq\,C\,|u(\tau)-v(\tau)|\,\left\|G(v)\right\|_{H^1_\ell}\,. \nonumber
\end{align}
Next, we need to bound in $L^2_{\ell-1}$ the quantity $\nabla\div\big(\oline a(\tau,x;u-v)\,G(v)\big)$: we have then to control four terms.
First of all, we notice that $\nabla\div\oline a\equiv0$. Moreover, we can write
\begin{align}  \label{est:dG_p2}
\left\|\div\oline a\;\nabla G(v)\right\|_{L^2_{\ell-1}}\,&\leq\,|u(\tau)-v(\tau)|\,\left\|\nabla G(v)\right\|_{L^2_{\ell-1}}\,,
\end{align}
and the same estimate holds true also for the term $\nabla\oline a\cdot\nabla G(v)$. Finally, we have
\begin{align}
\left\|\oline a\cdot\nabla^2G(v)\right\|_{L^2_{\ell-1}}\,&\leq\,C\,|u(\tau)-v(\tau)|\,\left\|\,\nabla^2G(v)\right\|_{L^2_\ell}\,. \label{est:dG_p3}
\end{align} 

Putting \eqref{est:dG_p1}, \eqref{est:dG_p2} and \eqref{est:dG_p3} together, we infer the bound
\begin{align*}
\left\|\div\big(\oline{a}(\tau,x;u-v)\,G(v)\big)\right\|_{H^1_{\ell-1}}\,&\leq\,C\,|u(\tau)\,-\,v(\tau)|\,\left\|G(v)\right\|_{H^2_\ell}\,.
\end{align*}
Inserting this last inequality into \eqref{est:dG_H^1_1} and using the bounds of Theorem \ref{th:weight}, we finally get the claimed estimate for the $H^1$-type norms of $\de G$.
\end{proof}

\subsection{Differentiability of the control-to-state map} \label{ss:G-diff}
 
 %Alfio: subsection spell - would be better sectio
In this section, we investigate differentiability properties of the control-to-state map $G$, defined above. 

Now, with Lemma \ref{l:Lip} at hand, we can prove G\^ateaux differentiability of $G$.
For any given $u$ in an open set $U_0\subset U_{ad}$, let $G(u)$ be the corresponding solution to the Liouville equation, as defined in Proposition \ref{p:G}, and let $\delta u=(\delta u_1,\delta u_2)$ be an admissible 
variation of $u$, such that $u+\veps \delta u \in U_{ad}$ for $\veps \in \R\setminus\{0\}$ sufficiently small. 
Then the G\^{a}teaux derivative of $G$ with respect to the variation $\delta u$ at $u$ is defined as the limit (whenever such a limit exists)
\begin{equation} \label{def:Gateaux}
\delta_{\delta u}G(u)\,:=\,\lim_{\veps \to 0} \frac{G(u+ \veps \delta u)\,-\,G(u)}{\veps}\,.
\end{equation}

The next proposition holds true.
\begin{prop} \label{p:Gateaux}
Let $m\geq2$ and $k\geq2$. Let the data $\rho_0$, $g$ and $a_0$ be fixed as in Remark \ref{r:data} above. Let $u$ belong to ${\rm int\,}U_{ad}$, where ${\rm int\,}U_{ad}$
denotes the interior part of the set $U_{ad}$.

Then, for any admissible variation $\de u$ of $u$, the limit \eqref{def:Gateaux} exists in $L^\infty_T(L^2)$. In particular, the control-to-state map $G$ is G\^ateaux differentiable
at $u$. Moreover, $\delta_{\delta u}G$ satisfies the Liouville problem 
\begin{equation}
\d_t\delta_{\delta u}G \, + \,  \div \bigl( a(t,x;u) \, \delta_{\delta u}G \bigr)\, = \, -\,\div\bigl( \oline a(t,x;\delta u) \, G(u) \bigr)\,,\qquad\qquad
\mbox{ with }\quad \delta_{\delta u} G_{|t=0}\,=\, 0\,,
\label{LiouvilleGateaux}
\end{equation}
where we have defined $\oline a(t,x;\delta u)\,:=\,\delta u_1 + x \circ \delta u_2$.
\end{prop}

\begin{proof}
For any $0<|\veps|<1$ small enough, let us define
$$
\de G^\veps\,:=\,\frac{1}{\veps}\,\big(G(u+ \veps \delta u)\,-\,G(u)\big)\,.
$$
Following the computations which led to \eqref{eq:dG}, we deduce that $\de G^\veps$ solves the equation
\begin{equation} \label{eq:de-G^a}
\d_t\de G^\veps\,+\,\div\big(a(t,x;u)\,\de G^\veps\big)\,=\,-\,\div\big(\oline a(t,x;\de u)\,G(u+\veps\de u)\big)\,,
\end{equation}
with initial datum $\de G^\veps_{|t=0}=0$.

%Alfio: observe replaced with notice
Notice that, by the first part of Lemma \ref{l:Lip}, we can find that, whenever $m\geq1$ and $k\geq1$, the sequence $\big(\de G^\veps\big)_\veps$ is uniformly bounded in $L^\infty_T(L^2)$.
Then, we can extract a subsequence, which converges weakly-$*$ in this space to some $\rho\in L^\infty_T(L^2)$. Then, by weak compactness methods (as the ones used in the proof
to Proposition \ref{p:ww-G}), we can pass to the limit in the weak formulation of the previous equation and gather that $\rho$ solves the problem
\begin{equation} \label{eq:Gat-rho}
\d_t\rho\,+\,\div\big(a(t,x;u)\,\rho\big)\,=\,-\,\div\big(\oline a(t,x;\de u)\,G(u)\big)
\end{equation}
with initial datum $\rho_{|t=0}=0$. Notice that the right-hand side of the previous equation belongs to $L^2$ by our assumptions on the initial data. Then, by uniqueness the whole
sequence $\big(\de G^\veps\big)_\veps$ has to weakly-$*$ converge to $\rho$, and $\rho$ has to coincide with $\de_{\de u}G$.

Unfortunately, the  previous argument does not prove the G\^ateaux differentiability of $G$, because we need that the limit exists in the strong topology, namely in the $L^\infty_T(L^2)$ norm.
In order to get this property, let us write the equation for $\de G^\veps-\rho$: we find
$$
\d_t\big(\de G^\veps-\rho\big)\,+\,\div\Big(a(t,x;u)\,\big(\de G^\veps-\rho\big)\Big)\,=\,-\,\div\Big(\oline a(t,x;\de u)\,\big(G(u+\veps\de u)-G(u)\big)\Big)\,,
$$
with zero initial datum. For notational simplicity, define $\rho^\veps\,:=\,\de G^\veps-\rho$; notice also that $G(u+\veps\de u)-G(u)\,=\,\veps\,\de G^\veps$.
Then, an energy estimate for that equation gives us
\begin{align*}
\left\|\rho^\veps(t)\right\|_{L^2}\,&\leq\,C\,\exp\left(C\int^t_0\|\div a(\tau,x;u)\|_{L^\infty}\right)\,
\int^t_0\left\|\div\Big(\oline a(\tau,x;\de u)\,\big(G(u+\veps\de u)-G(u)\big)\Big)\right\|_{L^2}\,d\tau \\
&\leq\,C\,\veps\,\exp\left(C\int^t_0\|\div a(\tau,x;u)\|_{L^\infty}\right)\,
\int^t_0|\de u(\tau)|\,\left(\left\|\de G^\veps\right\|_{L^2}\,+\,\left\|\big(1+|x|\big)\,\nabla\de G^\veps\right\|_{L^2}\right)\,d\tau \\
&\leq\,C\,\veps\,\exp\left(C\int^t_0\|\div a(\tau,x;u)\|_{L^\infty}\right)\,\int^t_0|\de u(\tau)|\,\left\|\de G^\veps\right\|_{H^1_1}\,d\tau\,,
\end{align*}
where we have argued as in the first line of \eqref{est:dG_L^2_part} in order to pass from the first inequality to the second one. At this point, applying the second estimate
of Lemma \ref{l:Lip} to equation \eqref{eq:de-G^a} yields
$$
\left\|\de G^\veps(\tau)\right\|_{H^1_1}\,\leq\,C_0\,\int^\tau_0|\delta u(s)|\,ds\,,
$$
for any $\tau\in[0,t]$, $t\leq T$, for a fixed constant $C_0$ (depending on $T$, $u^a$, $u^b$, and $\|\nabla a_0\|_{L^1_T(C^2_b)}$, $\|\rho_0\|_{H^2_2}$ and $\|g\|_{L^1_T(H^2_2)}$).
Putting this bound in the previous estimate entails
\begin{align*}
\left\|\rho^\veps(t)\right\|_{L^2}\,&\leq\,C\,C_0\,\veps\,\exp\left(C\int^t_0\|\div a(\tau,x;u)\|_{L^\infty}\right)\,\left(\int^t_0|\de u(\tau)|\,d\tau\right)^2 \\
&\leq\,C\,C_0\,\veps\,T^2\,\|\delta u\|^2_{L^\infty_T}\,\exp\left(C\int^T_0\|\div a(t,x;u)\|_{L^\infty}\right)\,.
\end{align*}
From this last estimate, we deduce that, in the limit for $\veps\ra0$, $\rho^\veps\,\longrightarrow\,0$ in $L^\infty_T(L^2)$. This completes the proof of the proposition.
\end{proof}

Next, we tackle the proof of the Fr\'echet differentiability of $G$. The arguments will follow the proof of Proposition \ref{p:Gateaux}.
\begin{theorem} \label{thm:G-Frechet}
Let $m\geq2$ and $k\geq2$. Let the data $\rho_0$, $g$ and $a_0$ be fixed as in Remark \ref{r:data} above, and let $u\in {\rm int\,}U_{ad}$.
Define $DG(u)[\delta u]$ to be the unique solution to equation \eqref{LiouvilleGateaux}.

Then there exists a constant $C>0$ (depending only on $T$, $u^a$, $u^b$, and $\|\nabla a_0\|_{L^1_T(C^2_b)}$, $\|\rho_0\|_{H^2_2}$ and $\|g\|_{L^1_T(H^2_2)}$) such that
\begin{align*}
\Big\|G(u+\delta u)\,-\,G(u)\,-\,DG(u)[\delta u]\Big\|_{L^\infty_T(L^2)}\,\leq\,C\,\standardNorm{\delta u}^2_{L^\infty_T}\,.
\end{align*}
In particular, the map $G$ is Fr\'echet differentiable from ${\rm int\,}U_{ad}$ into $L^\infty_T(L^2)$, and its Fr\'echet differential at any point $u\in {\rm int\,}U_{ad}$ is given by $DG(u)$.
\end{theorem}

\begin{proof}
In order to prove that $G$ is Fr\'echet differentiable, with Fr\'echet differential given by $DG(u)[\de u]$, we have to show that
\begin{align*}
\lim_{\standardNorm{\delta u}_{L^\infty_T} \rightarrow 0} \frac{\Big\|G(u+\delta u) - G(u) - DG(u)[ \delta u]\Big\|_{L^\infty_T(L^2)}}{\standardNorm{\delta u}_{L^\infty_T}}\,=\,0\,.
\end{align*}
We recall also that, if $G$ is Fr\'echet differentiable at $u$, then it is also G\^{a}teaux differentiable at the same point, and one has $\delta_{\delta u}G =DG(u)[ \delta u]$.

For simplicity, let us introduce the notation
$$
\mc G_u(\de u)\,:=\,G(u+\delta u) - G(u) - DG(u)[ \delta u]\,.
$$
Remark that, in the proof of Proposition \ref{p:Gateaux} above, we have already called $\rho$ the solution to equation \eqref{LiouvilleGateaux}, keep in mind equation \eqref{eq:Gat-rho}.
Therefore, the same computations performed on $\rho^\alpha\,=\,\de G^\alpha-\rho$ this time lead us to an equation for  $\mc G_u(\de u)$:
\begin{align*}
\d_t\mc G_u(\de u)\,+\,\div\big(a(t,x;u)\,\mc G_u(\de u)\big)\,=\,-\,\div\Bigl( \oline a(t,x;\delta u) \,\big( G(u+\delta u)-G(u) \big) \Big)\,,
\end{align*}
with initial datum $\mc G_u(\de u)_{|t=0}=0$.

Next, it is just a matter of repeating the estimates performed on $\rho^\alpha$: we easily find, for every $t\in[0,T]$, the inequality
\begin{align*}
\left\|\mc G_u(\de u)(t)\right\|_{L^2}\,&\leq\,C\,\exp\left(C\int^t_0\|\div a(\tau,x;u)\|_{L^\infty}\right)\,
\int^t_0\left\|\div\Big(\oline a(\tau,x;\de u)\,\big(G(u+\de u)-G(u)\big)\Big)\right\|_{L^2}\,d\tau \\
&\leq\,C\,\exp\left(C\int^t_0\|\div a(\tau,x;u)\|_{L^\infty}\right)\,\int^t_0|\de u(\tau)|\,\left\|G(u+\de u)-G(u)\right\|_{H^1_1}\,d\tau\,.
\end{align*}
Moreover, owing to the second inequality of Lemma \ref{l:Lip}, we get
$$
\left\|G(u+\de u)-G(u)\right\|_{H^1_1}\,\leq\,C_0\,\int^\tau_0|\delta u(s)|\,ds\,,
$$
for a constant $C_0$ which depends, as before, only on $T$, $u^a$, $u^b$, $\|\nabla a_0\|_{L^1_T(C^2_b)}$, $\|\rho_0\|_{H^2_2}$ and $\|g\|_{L^1_T(H^2_2)}$. Inserting
this relation in the previous estimate, we find
\begin{align*}
\left\|\mc G_u(\de u)(t)\right\|_{L^2}\,&\leq\,C\,C_0\,\exp\left(C\int^t_0\|\div a(\tau,x;u)\|_{L^\infty}\right)\,\left(\int^t_0|\de u(\tau)|\,d\tau\right)^2\,\leq\,K\,\|\de u\|^2_{L^\infty_T}\,,
\end{align*}
for a new positive constant $K$. From this last inequality, the claims of the theorem follow.
\end{proof}

\section{Analysis of the Liouville optimal control problem} \label{s:ocp}
%Alfio: small changes in text: 
In this section, we investigate our Liouville ensemble optimal control problem. In the first part, after recalling the problem's setting, we prove the existence of optimal controls by means
of classical arguments. However, notice that one has to carefully justify that the reduced functional  $\what J$ (see its definition below) is weakly lower semi-continuous. 
In fact, this property is not obvious, since $\rho=G(u)$ depends non-linearly on $u$. After that, in Section \ref{sec-OptimalitySystem} we characterise optimal controls as solutions
of a related first-order optimality system. In Section \ref{ss:unique} we discuss  uniqueness of optimal controls. 

\subsection{Existence of optimal controls} \label{sec-existencecontrol}

In this section, we deal with existence of optimal solutions to an \emph{ensemble optimal control problem}. Our analysis is based on the following assumptions.

\begin{itemize}
 \item[\textbf{(A.1)}] We fix $(m,k)\in\N^2$, %%%such that $m\geq 2$ and $k\geq2$ (\fra{such that $2\leq m\leq k$???}),
and we take an initial datum $\rho_0\in H^m_k(\R^d)$, a force
$g\in L^1\big([0,T];H^m_k(\R^d)\big)$ and a vector field $a_0\in L^1\big([0,T];C^{m+1}(\R^d)\big)$, with $\nabla a_0\in L^1\big([0,T];C^{m}_b(\R^d)\big)$.

\item[\textbf{(A.2)}] We fix parameters $(\g,\de,\nu)\in\R^3$ such that $\g>0$, $\de\geq0$ and $\nu\geq0$. %%%, with $\de+\nu>0$. \fra{Is the last condition necessary?}

\item[\textbf{(A.3)}] Chosen $u^a=\big(u^a_1,u^a_2\big)$ and $u^b=\big(u^b_1,u^b_2\big)$ in $\R^{2d}$, with $u^a\leq u^b$, we define the set of admissible controls to be
\begin{align}
U_{ad}\,&:=\,\left\{ u \,\in\, \LL^\infty_T(\RR^d)\;\bigl|\quad u^a\,\leq\,u(t)\,\leq\,u^b \qquad\mbox{ for a.e. }\; t\,\in\,[0,T]\right\}\qquad\mbox{ if }\quad \nu=0 \label{def:U_0} \\
U_{ad}\,&:=\,\left\{ u \,\in\, \HH^1_T(\RR^d)\;\bigl|\quad u^a\,\leq\,u(t)\,\leq\,u^b \qquad\mbox{ for all }\; t\,\in\,[0,T]\right\}\qquad\mbox{ if }\quad \nu>0\,. \label{def:U_nu}
\end{align}

\item[\textbf{(A.4)}] Finally, we take two 
attracting potentials $\theta$ and $\varphi$ in $L^2(\R^d)$, in the sense 
specified in Section \ref{sec-formulation}.
\end{itemize}

\begin{remark} \label{r:A4}
We point out that assumption \tbf{(A.4)} (which will be strengthened in Section \ref{ss:unique} for getting uniqueness, see condition \tbf{(A.4)*} there) is taken for simplicity of presentation,
since more general $\theta$ and $\vphi$ can be considered in our framework. For instance, we can allow for $\theta$ to depend on time: $\theta\,\in\,L^1_T(L^2)$, or $\theta\,\in\,L^1_T(H^1_1)$ in
\tbf{(A.4)*} below. 
The case $\theta(x)=|x|^2$ and $\varphi(x)=|x|^2$ is more delicate, and will be matter of further discussions in Section \ref{ss:confining}.
\end{remark}

Now, consider our cost functional given by 
\begin{align}
J(\rho,u)\,&:=\,\int_0^T\int_{\R^d}\theta(x)\,\rho(x,t)\,dx\,dt\, +\, \int_{\R^d} \varphi(x) \, \rho(x,T) \, dx \label{JfuncEnsemble} \\
&\qquad\qquad
+\frac{\gamma}{2}\,\int_0^T\big|u(t)\big|^2\,dt\,+\,\delta \,\int_0^T\big|u(t)\big|\,dt\,+\,\frac{\nu}{2}\,\int_0^T\left| \frac{d}{d t} u(t)\right|^2 \,dt\,. \nonumber
\end{align}
Remark that $J$ is well-defined whenever $u\in\LL^2_T$ if $\nu=0$, or $u\in\HH^1_T$ if $\nu>0$, and $\rho\in C\big([0,T];L^2(\R^d)\big)$.

Our ensemble optimal control problem requires to find
\begin{equation} \label{eq:min-J}
\min_{u \in U_{ad}} J(\rho,u)\,,
\end{equation}
subject to the differential constraint
\begin{equation} \label{LiouvilleEqEnsemble}
\left\{\begin{array}{ll}
       \d_t\rho\,+\,\div\bigl( a(t,x;u) \, \rho \bigr)\, =\,g \qquad\qquad & \mbox{ in } \quad [0,T]\times\R^d \\[1ex]
       \rho_{|t=0}\,=\, \rho_0 \qquad\qquad & \mbox{ on } \quad \R^d\,,
       \end{array}
\right.
\end{equation}
where the drift function $a(t,x;u)$ is defined as
\begin{equation} \label{def:a}
a(t,x;u)\,:=\,a_0(t,x)\,+\,u_1(t)\,+\, x \circ u_2(t)\,.
\end{equation}

Under our assumptions, Theorem \ref{th:weight} applies. Thus, for every $u\in U_{ad}$, there exists a unique solution solution $\rho \in C\big([0,T];H^{m}_k(\RR^d)\big)$ to the Liouville problem
\eqref{LiouvilleEqEnsemble} corresponding to that $u$. Therefore, resorting to the control-to-state map $G$, as defined in Section \ref{s:control-map},
we can introduce the so-called \emph{reduced cost functional}, given by 
\begin{equation} \label{def:J-red}
\what{J}(u)\,:=\,J\big(G(u),u\big)\,.
\end{equation}
Hence, the ensemble optimal control problem \eqref{eq:min-J}-\eqref{LiouvilleEqEnsemble} can be rephrased as follows:
\begin{equation} \label{minJred}
\min_{u \in U_{ad}}\what{J}(u)\,.
\end{equation}

\begin{remark} \label{r:J-red}
Recall that we have defined $G$ with values in $L^\infty_T(L^2)$. However, under our assumptions, we know that the solution to the Liouville equation actually belongs to
$C_T(L^2)$, so that the $\varphi$-term in \eqref{JfuncEnsemble} is well-defined, and 
thus so is $\what J$.
\end{remark}

In the following, we prove existence of a minimizer to  \eqref{minJred}.
\begin{theorem} \label{thm:existenceOptimalSolutions}
Under assumptions \tbf{(A.1)}-\tbf{(A.2)}-\tbf{(A.3)}-\tbf{(A.4)}, the ensemble optimal control problem \eqref{minJred} admits at least one solution 
$u^*\,\in\,U_{ad}$. The corresponding state $\rho^*\,:=\,G(u^*)$ belongs to the space $C\big([0,T];H^m_k(\R^d)\big)$.
\end{theorem}

\begin{proof}
Let us focus on the case $\nu=0$ for simplicity; the case $\nu>0$ follows from the same token.

As already mentioned, the functional $J$ given in \eqref{JfuncEnsemble} is well-defined for $(\rho,u)\,\in\,C_T(L^2)\times\LL^\infty_T$.
Now, we  remark that $U_{ad}$ is a bounded subset of $\LL^\infty_T$. On the other hand, owing to Theorem \ref{th:weight}, see especially
estimate \eqref{est:weight}, and the embedding $C_T(H^m_K)\hookrightarrow L^\infty_T(L^2)$, the map $G$ takes its values in a bounded set of $L^\infty_T(L^2)$.
It follows that $\what J$ is bounded; in particular, $\what J$ is a proper map, i.e. $\inf_{U_{ad}}\what J\,>\,-\infty$,
and $\what J$ is not identically equal to $+\infty$.

Next, we claim that $\what J$ is weakly lower semi-continuous. To prove this fact, it is enough to use the weak-weak continuity of $G$, as stated in Proposition \ref{p:ww-G},
and to remark that $J$ is weakly lower semi-continuous. Indeed, the last three terms in \eqref{JfuncEnsemble} are norms, so they are weakly lower semi-continuous. 
On the other hand, the first two terms are linear in $\rho$, and then they are weakly continuous with respect to the $L^\infty_T(L^2)$ and $L^2$ topologies, respectively.
Thus we immediately get that, if $\big(u_n\big)_n\subset U_{ad}$ is a sequence which converges weakly-$*$ to a $u\in U_{ad}$ in $L^\infty_T$, we have 
$$
\liminf_{n\ra+\infty}\what J(u_n)\,=\,\liminf_{n\ra+\infty}J\big(G(u_n),u_n\big)\,\geq\,J\big(G(u),u\big)\,=\,\what J(u)\,.
$$

At this point, proving the existence of a minimizer for $\what J$ is standard. Let us take a minimizing sequence $\big(u_n\big)_n\subset U_{ad}$.
Since $U_{ad}$ is a bounded set in $\LL^\infty_T$, we can extract a weakly-$*$ convergent subsequence, which we do not relabel for simplicity; let us call $u^*\in U_{ad}$ its limit-point.
Then, by the weak-lower semi-continuity of $\what J$, we can conclude that $u^*$ is a minimizer for $\what J$.
\end{proof}

We discuss uniqueness of the minimizers in Section \ref{ss:unique} below. 
For this purpose, we use characterization of minimizers as solutions to 
a suitable optimality system, which we derive in the next section.

\subsection{Liouville optimality systems} \label{sec-OptimalitySystem}

This section is devoted to the characterization of ensemble optimal controls as solutions of the related first-order optimality system.
For this purpose, in addition to hypotheses \tbf{(A.1)}-\tbf{(A.2)}-\tbf{(A.3)}-\tbf{(A.4)} stated above, from now on we take
$$
m\,\geq\,1 \qquad\qquad\mbox{ and }\qquad\qquad k\,\geq\,1\,.
$$

In correspondence to \eqref{JfuncEnsemble}-\eqref{eq:min-J}-\eqref{LiouvilleEqEnsemble}, we consider the Lagrange multipliers framework, see e.g. \cite{Lions1971,Troeltzsch2010},
and introduce the Lagrange functional $\LF$ as follows:
\begin{align} \label{def:LFunction}
\LF(\rho,u,q)\,:=\,J(\rho,u)+\int_0^T\!\!\!\SPI \Big( \d_t\rho+\div \big(a(x,t;u)\rho\big)-g\Big)q\,dxdt+\int_{\R^d}\big(\rho(0,x)-\rho_0(x)\big)q_0(x)\,dx\, , 
\end{align}
where, for the sake of generality, we have included a right-hand side $g$. 
The variable $q$ represents the Lagrange multiplier. Notice that $\LF$ is well-defined whenever $u\in\LL^\infty_T$ if $\nu=0$, $u\in\HH^1_T$ if $\nu>0$, $q\in L^\infty_T(L^2)$, $q_0\,\in\,L^2$
and $\rho\in C_T(L^2)$ such that both $\d_t\rho$ and $\div \big(a(x,t;u)\,\rho\big)$ belong to $L^1_T(L^2)$.
In particular, it is enough to have $\rho\,\in\,W^{1,1}_T(L^2)\,\cap\,L^\infty_T(H^1_1)$, recall also Proposition \ref{p:H^m_k}.
Notice that, \tsl{a posteriori}, we will find $q\in C_T(L^2)$ and $q_0=q(0)$; see the discussion below for details.

In order to derive the optimality system, let us discuss different instances.
For clarity, we first discuss the case with $L^2$ costs only, then the case with $L^2-H^1$ costs, and finally the case  with $L^2-L^1- H^1$ costs.

\paragraph{The case $\mbf{\de=\nu=0}$.}
If $\delta=0$, then $J$ is Fr\'echet differentiable over $C_T(L^2)\times {\rm int\,}U_{ad}$, since it is linear in $\rho$ and the 
control costs with $\gamma>0$, $\nu\geq 0$ are given by differentiable norms. 
It is then an easy computation to show that $\LF$ is Fr\'echet differentiable over the space
$$
\mbb{X}_T\;:=\;\left(W^{1,1}_T(L^2)\,\cap\,L^\infty_T(H^1_1)\right)\;\times\;\LL^2_T\;\times\;C_T(L^2)\,,
$$
where $\LL^2_T$ has to be replaced by $\HH^1_T$ in the case when $\nu>0$. The Fr\'echet differential of $\LF$ at $(\rho,u,q)$ is given by the linearization
of each of its terms at that point.

Now, consider in addition $\nu=0$. The \emph{optimality system} is obtained by putting to zero the Fr\'echet derivatives of $\LF(\rho,u,q)$ with respect to each of its arguments separately. We obtain 
\begin{align}
&\d_t\rho\,+\,\div\big(a(x,t;u)\,\rho\big)\,=\,g\,, \qquad\qquad \mbox{ with }\quad \rho_{|t=0}\,=\,\rho_0  \label{forwardEQ} \\
&-\,\d_tq\,-\,a(x,t;u)\cdot \nabla q\,=\,-\,\theta, \qquad\qquad \mbox{ with }\quad q_{|t=T}\,=\,-\,\varphi  \label{adjoint} \\
&\left( \gamma\,u^r_j\,+\,\int_{\RR^d}\,\div\left(\frac{\partial a}{\partial u^r_j}\,\rho\right)\, q\,dx\;,\;v^r_j\,-\,u^r_j \right)_{L^2(0,T)}\,\geq\,0\qquad
\forall v \in U_{ad}\,,\;j\,=\,1,2\,,\; r\,=\,1\ldots d\,. \label{reducedGradient}
\end{align}
We remark that, denoting by $e^r$ the $r$-th unit vector of the canonical basis of $\R^d$ and by $x^r$ the $r$-th component of the vector $x\in\R^d$, by Definition \ref{def:a} we have
$$
\frac{\partial a}{\partial u^r_j}(t,x;u)\,=\,e^r\qquad\mbox{ for }\;j=1\,,\qquad\qquad
\frac{\partial a}{\partial u^r_j}(t,x;u)\,=\,x^r\,e^r\qquad\mbox{ for }\;j=2\,.
$$
Then, equation \eqref{reducedGradient} can be equivalently written in the following form: for any $1\leq r\leq d$,
$$
\left\{\begin{array}{l}
        \left( \gamma\, u^r_1 \,+\,  \displaystyle{\int}_{\RR^d} \d_r\rho\,q\,dx\;,\;v^r_1\,-\,u^r_1 \right)_{L^2(0,T)}\,\geq\,0 \\[1ex]
        \left( \gamma\, u^r_2\,+\,\displaystyle{\int}_{\RR^d}\d_r\big(x^r\,\rho\big)\,q\,dx\;,\;v^r_2\,-\,u^r_2 \right)_{L^2(0,T)}\,\geq\,0\,.
       \end{array}
\right.
$$
Further, if we sum up equations \eqref{reducedGradient} for all $m$ and all $r$, we can write, in the following compact form
\begin{equation} \label{eq:e+x}
\left( \gamma\,u\,+\,\int_{\RR^d}\,\div\big((e+x)\,\rho\big)\, q\,dx\;,\;v\,-\,u \right)_{\LL^2_T}\,\geq\,0\qquad\qquad\mbox{ for all }\quad v\,\in\,U_{ad}\,,
\end{equation}
where we have defined the vector $e\,=\,(1\ldots1)$.

Equation \eqref{forwardEQ} is our Liouville model (also called the forward equation in this context). The results of Section \ref{ss:weight} guarantee that, under
our assumptions, there exists a unique solution $\rho\,\in\,C_T(H^1_1)$. Moreover, since $u\,\in\,U_{ad}$, an inspection of \eqref{forwardEQ} reveals
that $\d_t\rho\,\in\,L^1_T(L^2)$. 

Equation \eqref{adjoint} is the adjoint Liouville equation; it is obtained by taking the 
Fr\'echet derivative of \eqref{def:LFunction} with respect to $\rho$. This is a transport equation that evolves 
backwards in time. By setting $\wtilde{q}(t,x)\,=\,q(T-t,-x)$, we obtain a transport  problem for $\wtilde{q}$, as in \eqref{eq:transportProblem}, with
source term $-\theta$ and initial condition $\wtilde q_{|t=0}\,=\,-\varphi$. 
Thus, the results of Paragraph \ref{sss:transport} guarantee the existence and uniqueness of a Lagrange multiplier $q \in C_T(L^2)$,
provided that $\theta$ and $\varphi$ are in $L^2$.

From the discussion above, we get that any solution to the optimality system \eqref{forwardEQ}-\eqref{adjoint}-\eqref{reducedGradient}, with $u\in U_{ad}$, belongs indeed to the space $\mbb{X}_T$.

\medbreak
Equation \eqref{reducedGradient} represents the optimality condition. To better illustrate this fact, we suppose from now on that 
$$
m\,\geq\,2 \qquad\qquad\mbox{ and }\qquad\qquad k\,\geq\,2\,.
$$
Then, the reduced cost functional $\what J$, defined in \eqref{def:J-red}, is Fr\'echet differentiable; in terms of the reduced minimisation problem \eqref{minJred}, the optimal solution $u^*$ in 
the convex, closed and bounded set $U_{ad}$ is characterized by the optimality condition given by 
$$
\left( \nabla_u\what{J}(u^*)\;,\;  \, v\,-\,u^*  \right)_{\LL^2_T} \ge 0 , \qquad \text{ for all } v \in U_{ad}, 
$$
where $\nabla_u \what{J}$ denotes the $L^2$-gradient of $\what{J}$ with 
respect to $u$. In fact, a direct computation of $\nabla_u J\big(G(u),u\big)$, with the 
introduction of the auxiliary adjoint variable $q$, gives the optimality 
system above, and the following relation:
$$
\nabla_{u^r_j} \what{J}(u)\,= \,\gamma\, u^r_j\, +\,  \int_{\RR^d} \div\left(\frac{\partial a}{\partial u^r_j}\, \rho\right)\, q\,dx\,.
$$

\paragraph{The case $\mbf{\de=0\,,\;\nu>0}$.}
Next, assume that $\delta =0$ and $\gamma, \nu >0$. Recall that, in this case, the set $U_{ad}$ is defined by \eqref{def:U_nu}.
Then, the natural Hilbert space where $u^*$ is sought is $\wtilde \HH^1_T(\RR^d)\,:=\, \wtilde H^1_T(\RR^d) \times \wtilde H^1_T(\RR^d)$, where $\wtilde H^1_T$ corresponds to the $H^1_T$ space,
endowed with the weighted $H^1$-product given by
$$
(u,v)_{\wtilde H^1_T}\, :=\,\gamma \, \int_0^T u(t) \cdot v(t) \, dt\, +\, \nu \, \int_0^T u'(t) \cdot v'(t) \, dt \,.
$$
The notation $\phantom{u}'\,=\,d/dt$ stands for the weak time derivative.  

Now, let $\mu$ be the $\wtilde H^1$-Riesz representative of the continuous linear functional
$$
v\;\mapsto\; \left(   \int_{\RR^d} \div\left(\frac{\partial a}{\partial u} \,\rho\right) q\,dx\;,\;v \right)_{\LL^2_T}\,.
$$
Assuming that $u\,\in\,U_{ad}\cap H^1_0\big([0,T];\R^{2d}\big)$, then $\mu$ can be computed by solving the equation
\begin{equation}
	\left(-\, \nu \,\frac{d^2}{dt^2}\,+\,\gamma\right)\mu\,=\,\int_{\RR^d}\div\left(\frac{\partial a}{\partial u}\,\rho \right)\,q\,dx\,, \qquad \mu(0)\,=\,\mu(T)\,=\,0\,,
	\label{ellEqmu}
\end{equation}
which is understood in a weak sense. Notice that the choice $u\in H^1_0\big([0,T];\R^{2d}\big)$ corresponds to the modelling requirement that the control is switched on at $t=0$ and 
switched off at $t=T$. Other initial and final time conditions on $u$ may be required 
and encoded as boundary conditions in \eqref{ellEqmu}. 

With the setting above, the $\wtilde H^1$-gradient is given, for $j=1,2$ and all $r=1\ldots d$, by 
\begin{equation} \label{eq:H^1-gradient}
\wtilde \nabla_{u^r_j} \what{J}(u)\,=\, u^r_j\, +\, \mu^r_j\,.
\end{equation}
The optimality condition \eqref{reducedGradient} then becomes 
\begin{equation} \label{eq:opt-cond_H^1}
\big( u^r_j\, +\,  \mu^r_j \;,\;  v^r_j\,-\,u^r_j \big)_{\wtilde{H}^1_T}\, \geq\, 0
\end{equation}
for all $v\,\in\,U_{ad}$, where $U_{ad}$ is given in \eqref{def:U_nu}, $j=1,2$ and $1\leq r\leq d$.

\paragraph{The case $\mbf{\de>0}$.}
In this case, a $L^1$ norm of the control appears in the cost functional. This term is not G\^{a}teaux differentiable and the discussion becomes more involved.
By using of the control-to-state map, we start by defining
\begin{align*}
f(u)\,&:=\,\int_0^T\!\!\! \int_{\R^d} \theta(x)  \, G(u)(x,t) \, dx\, dt\, +\, \int_{\R^d} \varphi(x) \, G(u)(x,T) \, dx\,+\,\frac{\gamma}{2} \int_0^T \big|u(t)\big|^2 \, dt\,+\, 
\frac{\nu}{2}\int_0^T  \left|\frac{d}{d t} u(t)\right|^2 \, dt  \\
g(u)\,&:= \,\delta \, \standardNorm{u}_{L^1_T}\,.
\end{align*}
The $L^1$-cost, represented by $g$, admits a subdifferential $\partial g(u)\, =\, \delta \, \partial\big(\standardNorm{u}_{L^1}\big) $,
see e.g. Section 2.3 of \cite{Barbu12}. If we denote by $\LL_T^*\,:=\,\big(\LL^\infty_T(\RR^d)\big)^*$ and by $\lan\cdot,\cdot\ran$ the duality product in $\LL^*_T\times\LL^\infty_T$,
the following formula holds true:
	\begin{align}
	\partial\big(\standardNorm{u}_{L^1}\big)\,&= \,\Big\{ \phi \in \LL^*_T\;\big|\quad
	\standardNorm{v}_{L^1}\, -\, \standardNorm{u}_{L^1}\,\geq \,\big\lan\phi\,,\,v-u\big\ran\quad \forall\, v \in U_{ad}  \Big\}  \label{eq:subgradientDelta} \\
	&=\,\begin{cases}
		\phi \in \LL^*_T\;\big|\quad \standardNorm{\phi}_{\LL^*_T}\, =\,1\,,\;\phi(u)\,=\, \standardNorm{u}_{\LL^\infty_T} &\mbox{ if }\quad u\not\equiv 0 \\[1ex]
		\text{unit ball in }\;\LL^*_T &\mbox{ if }\quad u\equiv 0 
	\end{cases} \Bigg\}\,. \nonumber
	\end{align}

Now, the reduced functional can be written as
%\begin{align}
$\what{J}(u)\,=\, f(u)\, +\, g(u)$.
%\end{align}
In this case, the equations \eqref{forwardEQ} and \eqref{adjoint} in the corresponding optimality system are the same;
however, we have a different optimality condition \eqref{reducedGradient}. In the case $\nu=0$, as in Theorem 2.2 in \cite{CiaramellaBorzi16}, we have the following result; for its proof, we refer
to \cite{CiaramellaBorzi16} and \cite{Stadler07}. Notice that, as for equations \eqref{forwardEQ}-\eqref{adjoint}-\eqref{reducedGradient}, equation \eqref{eq:optimalitySystemDelta} below
can be written even when $G$, and hence $\what J$, are not Fr\'echet differentiable.
\begin{theorem}
\label{thm:OptimalitySystemSSN}
Under assumptions \tbf{(A.1)}-\tbf{(A.2)}-\tbf{(A.3)}-\tbf{(A.4)}, where we take $m\geq1$ and $k\geq1$, suppose moreover that the pair $(\rho,u)\,\in\,C_T(H^m_k)\times U_{ad}$
is a minimizer for \eqref{minJred}.

Then there exists a unique $q\,\in\,C_T(L^2)$ which solves \eqref{adjoint}, and a $\what{\lambda}\,\in \,\partial g(u)$ such that the following inequality condition is satisfied:
\begin{align}
\left( \gamma\,u^r_j\, +\, \what{\lambda}_j^r \,+\, \SPI \div\left(\frac{\partial a}{\partial u^r_j}\,\rho\right) q\,dx\;,\; v^r_j-u^r_j \right)_{L^2(0,T)}\,\geq\,0\quad
\forall \,v \in U_{ad}\,,\;j= 1,2\,,\;r = 1\ldots d \,.
\label{eq:optimalitySystemDelta}
\end{align}
Moreover, there exist $\lambda_+$ and $\lambda_-$, belonging to $L^ \infty_T(\RR^d)$, such that \eqref{eq:optimalitySystemDelta} is equivalent to the equations
$$%\begin{subnumcases}{}
\left\{\begin{array}{l}
\gamma\, u^r_j\, + \, \displaystyle{\SPI} \div\left(\dfrac{\partial a}{\partial u^r_j}\, \rho \right) q\,dx\,+\,(\lambda_+)_j^r-(\lambda_-)_j^r+\what{\lambda}_j^r\,=\, 0 \\ %\label{eq:redGradientDeltaa} \\
(\lambda_+)_j^r\,\geq\, 0\,,\qquad u^b-u_j^r\, \geq\, 0\,,\qquad (\lambda_+)_j^r\;(u^b-u_j^r) \,=\, 0 \\[1ex]
(\lambda_-)_j^r \,\geq\, 0\,,\qquad u_j^r - u^a\, \geq\, 0\,,\qquad (\lambda_-)_j^r\;(u_j^r-u^a)\, =\, 0  \\[1ex]
\what{\lambda}_j^r\, =\, \delta\qquad \mbox{ a.e. in }\quad\left\{t \in [0,T]\;\big|\quad u_j^r(t)\,>\,0 \right\} \\[1ex]
\left|\what{\lambda}_j^r\right|\,\leq\,\delta\qquad \mbox{ a.e. in }\quad\left\{t \in [0,T]\;\big|\quad u_j^r(t)\,=\,0 \right\} \\[1ex]
\what{\lambda}_j^r\, =\, \delta\qquad \mbox{ a.e. in }\quad\left\{t \in [0,T]\;\big|\quad u_j^r(t)\,<\,0 \right\}\,,
       \end{array} \right.
%\label{eq:redGradientDeltaf}
$$%\end{subnumcases}
for $j=1,2$ and all $1\leq r\leq d$.
\end{theorem}

\begin{remark} \label{r:OptSyst-delta}
In our case, $\what{\lambda}_j^r$ can be understood to be $\de\,\sgn(u_j^r)$, where $\sgn{(x)}$ is the sign function, equal to $1$ or $-1$ depending if $x>0$ or $x<0$ respectively, and equal to $0$ if $x=0$.

Furthermore, we notice that the additional Lagrange multipliers $(\lambda_\pm)_j^r$ are due to the constraints $u^a\,\leq\, u(t)\, \leq\, u^b$  for almost all $t \in [0,T]$.
\end{remark}

Finally, the case $\delta>0$ and $\nu>0$ can be treated as done before. After resorting once again to the space $\wtilde \HH^1_T$, let 
$\mu$ be the $\wtilde H^1$-Riesz representative of the continuous linear functional
$$
v\;\mapsto\; \left( \what{\lambda }\, +\, \int_{\RR^d} \div\left(\frac{\partial a}{\partial u}\,\rho\right) q\,dx\;,\;v\right)_{\LL^2_T}\,.
$$
Then, assuming that $u\in U_{ad}\cap H^1_0\big([0,T];\R^{2d}\big)$, we can compute $\mu$ as above, by solving the equation
$$ %\label{ellEqmu2}
\left(-\, \nu \,\frac{d^2}{dt^2}\,+\,\gamma\right)\mu\,=\,\what\lambda\,+\,\int_{\RR^d}\div\left(\frac{\partial a}{\partial u}\,\rho \right)\,q\,dx\,, \qquad \mu(0)\,=\,\mu(T)\,=\,0\,,
$$
which has to be understood again in a weak sense. With this definition, relation \eqref{eq:H^1-gradient} still holds true, and the optimality condition \eqref{reducedGradient} can be expressed
once again by equations \eqref{eq:opt-cond_H^1}.

\subsection{Uniqueness of optimal controls} \label{ss:unique}

In this section, we tackle the problem of uniqueness of optimal controls. 
Specifically, we are able to prove uniqueness in the situation when $\de=0$ and $\nu=0$ in \eqref{JfuncEnsemble}. Indeed, our proof relies on the characterization of optimal controls as solutions to
the corresponding optimality system, and especially on the use of the optimality condition \eqref{reducedGradient}. For this reason, the cases $\de>0$ or $\nu>0$ read more complicated,
recall equation \eqref{eq:opt-cond_H^1} and Theorem \ref{thm:OptimalitySystemSSN} above, and are left aside in our discussion.

%Alfio: first of all :-)
We start with discussing uniqueness for the unconstrained-control problem. Then, in Section \ref{sss:u-constr} we prove uniqueness for the constrained problem, under some additional assumptions.

\subsubsection{Uniqueness for the unconstrained-control problem} \label{sss:u-unconstr}

We consider the unconstrained-control optimization problem, with the assumptions \tbf{(A.1)}-\tbf{(A.2)}-\tbf{(A.4)}.

The first problem to deal with is then to prove the existence of a global minimizer: it is apparent that the proof of Theorem \ref{thm:existenceOptimalSolutions} does not work anymore,
since it relies on the uniform boundedness of the set $U_{ad}$ in an essential way.
For solving this issue, first of all we have to quit the $\LL^\infty_T$ framework, and rather work with controls which are merely in $\LL^2_T$; this is however quite natural,
in view of the form of \eqref{JfuncEnsemble}.

%Alfio: first of all

Moreover, we need additional assumptions: we require that both $\rho_0$ and $g$ are bounded from below
by some constant, that we can suppose, without loss of generality, to be $0$ (a case which is physically relevant, recall the discussion in Section \ref{sec-DynModelsLiouville}).
Moreover, we need also a lower bound on the optimization functions $\theta$ and $\varphi$: defined $c\in\R$ such a lower bound, up to working with $\theta-c$ and $\varphi-c$ and taking
$k$ large enough in \tbf{(A.1)}, so that both $\rho_0$ and $g$ belong to $L^1$ (and so does $\rho$, then), we can assume that also $c=0$.
 
Our additional hypotheses then read as follows.
\begin{itemize}
 \item[\tbf{(A.5)}] Suppose that $\rho_0\geq0$ and $g\geq0$.
 
 \item[\tbf{(A.6)}] We assume also that $\theta\geq 0$ and $\varphi\geq 0$.
\end{itemize}

Under these hypotheses, we are able first of all to prove the existence of a global minimizer.
\begin{prop} \label{p:unconstr-E}
Under assumptions \tbf{(A.1)}-\tbf{(A.2)}-\tbf{(A.4)}-\tbf{(A.5)}-\tbf{(A.6)} and the conditions $\de=\nu=0$, the unconstrained ensemble optimal control problem
$$
%\min_{u\in\LL^\infty_T}\what J(u)\qquad \mbox{ if }\quad \nu=0\,,\qquad\qquad\qquad
\min_{u\in\LL^2_T}\what J(u) %\qquad \mbox{ if }\quad \nu>0\,,
$$
has at least one solution $u^*\,\in\,\LL^2_T$.
\end{prop}

\begin{proof}
By assumption \tbf{(A.5)}, it is standard to deduce that any $\rho$ solution to \eqref{LiouvilleEqEnsemble}, for any fixed $u$, is positive almost everywhere on $[0,T]\times\R^d$.
Then, if moreover \tbf{(A.6)} holds, then $\what J(u)\,\geq\,0$ for any $u\in\LL^2_T$, so $\what J$ is in particular a proper map.

Next, we claim that $\what J$ is a coercive map. Indeed, let $\big(u_n\big)_n$ be a sequence in the space $\LL^2_T$, such that $\left\|u_n\right\|_{\LL^2_T}\,\longrightarrow\,+\infty$
for $n\ra+\infty$. Then, by \eqref{JfuncEnsemble}, we have that $\what J(u_n)$ must explode to $+\infty$.

Finally, proving the weak lower semi-continuity of $\what J$ requires just small adaptations to the proof of Theorem \ref{thm:existenceOptimalSolutions}.
In fact, it is easy to see that Proposition \ref{p:ww-G} holds true even if we replace the weak-$*$ convergence in $\LL^\infty_T$ with the weak convergence in $\LL^2_T$
(this is due to the fact that $W^{1,1}_T$ is compactly embedded into $L^q_T$ for any $1\leq q<+\infty$).

With these ingredients at hand, and after remarking that the coercivity of $\what J$ implies that any minimizing sequence has to remain uniformly bounded
in $\LL^2_T$, showing the existence of a minimizer also follows the same lines of the proof to Theorem \ref{thm:existenceOptimalSolutions}.
\end{proof}

In order to prove uniqueness, we need additional regularity on the cost functions $\theta$ and $\varphi$. We then formulate the following assumption, which strengthen \tbf{(A.4)}.
\begin{itemize}
 \item[\tbf{(A.4)*}] Suppose that both $\theta$ and $\varphi$ belong to $H^1_1(\R^d)$.
\end{itemize}

The main result of this section reads as follows.
\begin{theorem} \label{thm:opt-contr_u}
Under assumptions \tbf{(A.1)}-\tbf{(A.2)}-\tbf{(A.5)}-\tbf{(A.6)}-\tbf{(A.4)*}, suppose also that both $m\geq 2$ and $k\geq2$. In addition, take $\de=\nu=0$ in \eqref{JfuncEnsemble}.

Then, the optimal control $u^*$, whose existence is guaranteed by Proposition \ref{p:unconstr-E}, is unique in $\LL^2_T$.
\end{theorem}

\begin{proof}
Let $(u_1,\rho_1,q_1)$ and $(u_2,\rho_2,q_2)$ be two optimal controls with corresponding state and adjoint state. Then both triplets have to satisfy the optimality system
\eqref{forwardEQ}-\eqref{adjoint}, together with the equality
$$
\g\,u\,+\,\int_{\R^d}\div\big((e+x)\,\rho\big)\,q\,dx\,=\,0\,,
$$
which replaces the optimality condition \eqref{reducedGradient}, since now the problem is unconstrained.

We now apply the previous equality to each of the optimal triplets and take the difference: after setting $\de u=u_1-u_2$ and analogous notations for $\de\rho$ and $\de q$, we find
$$
\g\,\de u\,+\,\int_{\R^d}\div\big((e+x)\,\de\rho\big)\,q_1\,dx\,+\,\int_{\R^d}\div\big((e+x)\,\rho_2\big)\,\de q\,dx\,=\,0\,,
$$
which immediately implies, for almost every $t\in[0,T]$, the estimate
\begin{align} 
\g\,\big|\de u(t)\big|\,&\leq\,\SPI\left|\div\big((e+x)\,\de\rho(t)\big) q_1(t)\right|+\SPI\left|\div\big((e+x)\,\rho_2(t)\big)\,\de q(t)\right|\,. \label{est:de-u}
\end{align}

%Alfio: former term is not clear: first ?
Let us now estimate each one of the integral terms. We start with the former term 
and obtain
\begin{align*}
\SPI\left|\div\big((e+x)\,\de\rho(t)\big)\,q_1(t)\right|\,dx\,&\leq\,\left\|q_1(t)\right\|_{L^2}\,\left\|\div\big((e+x)\,\de\rho(t)\big)\right\|_{L^2} \\
&\leq\,C_1\,\left(\left\|\de\rho(t)\right\|_{L^2}\,+\,\left\|\big(1+|x|\big)\,\nabla\de\rho(t)\right\|_{L^2}\right)\;\leq\,C_1\,\left\|\de\rho(t)\right\|_{H^1_1}\,,
\end{align*}
where, in passing from the first to the second inequality, we have computed explicitly the derivatives in the $\div$ term, and we have used Theorem \ref{thm:ex-u_Tr}
applied to the transport equation \eqref{adjoint} for treating the $q_1$ term.
Notice that the constant $C_1$ can be expressed as 
\begin{equation} \label{def:C_1}
C_1\,:=\,C\,\exp\Big(C\left(\left\|\div a_0\right\|_{L^1_T(L^\infty)}+\left\|u_1\right\|_{\LL^1_T}\right)\Big)
\big(\left\|\varphi\right\|_{L^2}\,+\,T\,\left\|\theta\right\|_{L^2}\big)\,,
\end{equation}
for a ``universal'' constant $C>0$ that depends on the space dimension $d$.
At this point, we recall that both $\rho_1$ and $\rho_2$ satisfy equation \eqref{forwardEQ}, with controls $u_1$ and $u_2$, respectively. Then, taking their difference and
applying Lemma \ref{l:Lip} finally yields, for a new constant $\wtilde{C}_1\,=\,C_1\,K_1^{(2)}$ just depending on the data of the problem, the following bound:
\begin{equation} \label{est:de-u_int1}
\SPI\left|\div\big((e+x)\,\de\rho(t)\big) q_1(t)\right|\,dx\,\leq\,\wtilde{C}_1\,\int^t_0\big|\de u(\tau)\big|\,d\tau\,.
\end{equation}

Next, consider the second integral in \eqref{est:de-u}. The computations are similar to the previous ones: first of all, we can estimate
\begin{align*}
\SPI\left|\div\big((e+x)\,\rho_2(t)\big)\,\de q(t)\right|\,dx\,&\leq\,\left\|\de q(t)\right\|_{L^2}\,\left\|\div\big((e+x)\,\rho_2(t)\big)\right\|_{L^2} \\
&\leq\,\left\|\de q(t)\right\|_{L^2}\,\left\|\rho_2(t)\right\|_{H^1_1}\;\leq\,C_2\,\left\|\de q(t)\right\|_{L^2}\,,
\end{align*}
where this time we have applied Theorem \ref{th:weight} to equation \eqref{forwardEQ} for $\rho_2$ to control its $H^1_1$ norm. In particular, it follows from that theorem that
\begin{equation} \label{def:C_2}
C_2\,:=\,C\,\exp\left(C\left(\left\|\nabla a_0\right\|_{L^1_T(C^1_b)}\,+\,\left\|u_2\right\|_{\LL^1_T}\right) \right)\,
\left(\standardNorm{\rho_0}_{H^1_1}\, +\, \standardNorm{g}_{L^1_T(H^1_1)} \right)\,,
\end{equation}
for a ``universal'' constant $C>0$.

As done above, we use now the fact that $q_1$ and $q_2$ are both solutions of \eqref{adjoint}, related to the controls $u_1$ and $u_2$ respectively. Hence, taking the difference of those equations and
arguing as in the proof of Lemma \ref{l:Lip} (keep in mind also Remark \ref{r:tr-weight}), one easily infers the existence of a ``universal'' constant $C>0$ such that
\begin{align*} 
\left\|\de q(t)\right\|_{L^2}\,&\leq\,C\,\exp\Big(C\left(\left\|\div a_0\right\|_{L^1_T(L^\infty)}\,+\,\left\|u_1\right\|_{\LL^1_T}\right)\Big)\,
\int^t_0|\de u(\tau)|\,\left\|\big(1+|x|\big)\,\nabla q_2(\tau)\right\|_{L^2}\,d\tau \\ %\label{est:de-q_L^2}
&\leq\,C\,\exp\Big(C\left(\left\|\nabla a_0\right\|_{L^1_T(C^1_b)}+\left\|u_1\right\|_{\LL^1_T}+\left\|u_2\right\|_{\LL^1_T}\right)\Big)
\left(\left\|\varphi\right\|_{H^1_1}+T\,\left\|\theta\right\|_{H^1_1}\right)\int^t_0\big|\de u(\tau)\big|\,d\tau\,. \nonumber
\end{align*}
After defining the constants
\begin{equation} \label{def:K-q}
\wtilde{K}^{(1)}_1\,:=\,C\,\exp\Big(C\left(\left\|\nabla a_0\right\|_{L^1_T(C^1_b)}+\left\|u_1\right\|_{\LL^1_T}+\left\|u_2\right\|_{\LL^1_T}\right)\Big)
\left(\left\|\varphi\right\|_{H^1_1}+T\,\left\|\theta\right\|_{H^1_1}\right)
\end{equation}
and $\wtilde{C}_2\,:=\,C_2\,\wtilde{K}^{(1)}_1$, we obtain
\begin{equation} \label{est:de-u_int2}
\SPI\left|\div\big((e+x)\,\rho_2(t)\big)\,\de q(t)\right|\,dx\,\leq\,\wtilde{C}_2\,\int^t_0\big|\de u(\tau)\big|\,d\tau\,.
\end{equation}

At this point, we can insert estimates \eqref{est:de-u_int1} and \eqref{est:de-u_int2} into \eqref{est:de-u}, and get, for a new constant $K\,=\,\wtilde C_1+\wtilde C_2$, the relation
$$
\g\,\big|\de u(t)\big|\,\leq\,\,K\,\int^t_0\big|\de u(\tau)\big|\,d\tau\,.
$$
An application of Gr\"onwall's lemma hence gives that $\de u\,\equiv\,0$ almost everywhere on $[0,T]$. This concludes the proof of the theorem.
\end{proof}

\subsubsection{The constrained optimization problem} \label{sss:u-constr}
%Alfio: too many now
In this section, we consider our optimization problem with constraints on the control $u \in U_{ad}$, where the call $U_{ad}$ is defined by \eqref{def:U_0}.
For the reasons explained above, we still restrict to the case $\de=\nu=0$.

In the constrained-control case, the characterization of optimal controls is given by an inequality, see \eqref{reducedGradient}. This is a very weak information: this is
the reason why we are able to prove uniqueness
only under a smallness condition, either on the time $T$ or on the size of the data $\rho_0$, $g$, $\nabla a_0$, $\theta$ and $\varphi$ in their respective functional spaces.

Let us recall that existence of an optimal control has been proved in Theorem \ref{thm:existenceOptimalSolutions} above. We can now state our uniqueness result.
\begin{theorem} \label{thm:opt_u-constr}
Under assumptions \tbf{(A.1)}-\tbf{(A.2)}-\tbf{(A.3)}-\tbf{(A.4)*}, suppose that both $m\geq 2$ and $k\geq2$. Take moreover $\de=\nu=0$ in \eqref{JfuncEnsemble}.
%Alfio: finally - spell
Finally, define
$$
\wtilde{K}\,:=\,C\,\exp\Big(C\left(\left\|\nabla a_0\right\|_{L^1_T(C^2_b)}+T\,\max\left\{\big|u^a\big|\,,\,\big|u^b\big|\right\}\right)\Big)\,
\left(\left\|\rho_0\right\|_{H^2_2}\,+\,\left\|g\right\|_{L^1_T(H^2_2)}\right)\,\left(\left\|\varphi\right\|_{H^1_1}\,+\,T\,\left\|\theta\right\|_{H^1_1}\right),
$$
where the constant $C>0$ can be taken as the maximum of the constants $C$ appearing in \eqref{def:C_1}, \eqref{def:C_2}, \eqref{def:K-q} and in the definition \eqref{def:K_1} of $K^{(2)}_1$.

If the condition
$$
\frac{\wtilde{K}\,T}{\g}\,<\,2 
$$
holds true, then there exists at most one optimal control $u^*$ in ${\rm int\,}U_{ad}$.
\end{theorem}

\begin{proof}
As above, let $(u,\rho_1,q_1)$ and $(v,\rho_2,q_2)$ be two optimal triplets solving the minimization problem \eqref{minJred}. Then both have to satisfy the optimality system
\eqref{forwardEQ}-\eqref{adjoint}-\eqref{reducedGradient}. In particular, the inequality in \eqref{reducedGradient}, written in the more explicit form \eqref{eq:e+x}, gives
\begin{align*}
\left(\gamma\,u\, +\, \SPI \div\big((e+x)\rho_1\big)\,q_1\;,\;u\,-\,w\right)_{\LL^2_T}\,&\leq\,0\qquad\qquad & \forall\,w \in U_{ad} \\
\left(\gamma\, v\, +\, \SPI \div\big((e+x)\rho_2\big)\,q_2\;,\;w-v\right)_{\LL^2_T}\,&\geq\, 0\qquad\qquad &\forall\, w \in U_{ad}\,.
%\label{uniqueness:u1}
\end{align*}
Apply the former inequality with $w=v$ and the latter with $w=u$, and subtract the two obtained relations: if we set $\de\rho\,:=\,\rho_1-\rho_2$ and $\de q\,:=\,q_1-q_2$, we get
\begin{align*}
\left(\gamma\,(u-v)\,+\,\SPI\div\big((e+x)\,\de\rho\big) q_1\,dx\,+\,\SPI \div\big((e+x)\,\rho_2\big)\,\de q\,dx\;,\;u-v\right)_{\LL^2_T}\,\leq\,0\,.
\end{align*}
From the previous inequality, straightforward computations allow to deduce that
\begin{align} 
\g\int^T_0|u(t)-v(t)|^2\,dt\,&\leq\,-\int^T_0\!\!\left(\SPI\!\!\div\big((e+x)\,\de\rho\big) q_1\,+\,\SPI\!\!\div\big((e+x)\,\rho_2\big)\,\de q\right)\big(u(t)-v(t)\big)\,dt \label{est:u-v} \\
&\leq\,\int^T_0\!\!\left(\SPI\left|\div\big((e+x)\,\de\rho\big) q_1\right|+\SPI\left|\div\big((e+x)\,\rho_2\big)\,\de q\right|\right)\,|u(t)-v(t)|\,dt\,. \nonumber
\end{align}

%Alfio: absorb - spell

At this point, one may want to apply a Cauchy-Schwarz inequality with respect to time, and a Young inequality to absorb the second term on the left-hand side. This is certainly possible;
however, in order to optimize the value of the constants, it is better to directly estimate the integral terms, at any time $t\in[0,T]$. Such a bound can be performed exactly as in the
proof of Theorem \ref{thm:opt-contr_u} above: we then find again inequalities \eqref{est:de-u_int1} and \eqref{est:de-u_int2}. Inserting them into \eqref{est:u-v} yields
$$
\g\int^T_0\big|u(t)-v(t)\big|^2\,dt\,\leq\,K\,\int^T_0\big|u(t)-v(t)\big|\,\left(\int^t_0\big|u(s)-v(s)\big|\,ds\right)\,dt\,,
$$
where $K$ is defined as in the proof of Theorem \ref{thm:opt-contr_u} above. 

For simplicity of notation, define $\s(t)\,:=\,\big|u(t)-v(t)\big|$ for almost every $t\in[0,T]$. The previous estimate becomes then
$$
\g\int^T_0\big(\s(t)\big)^2\,dt\,\leq\,K\int^T_0\s(t)\,\left(\int^t_0\s(s)\,ds\right)dt\,=\,\frac{K}{2}\left(\int^T_0\s(t)\,dt\right)^{\!2}\,.
$$
Using a Cauchy-Schwarz inequality for the term on the right-hand side, we finally get
$$
\g\int^T_0\big(\s(t)\big)^2\,dt\,\leq\,\frac{K\,T}{2}\int^T_0\big(\s(t)\big)^2\,dt\,,
$$
which obviously implies $\s\equiv0$ almost everywhere on $[0,T]$ whenever $K\,T/\g\,<\,2$. Then, we conclude the proof remarking that $K\,\leq\,\wtilde{K}$.
\end{proof}

\subsection{The case of confining $\theta$ and $\varphi$} \label{ss:confining}
 
As pointed out in Remark \ref{r:A4}, from the applications viewpoint, it 
may be desirable to consider the case when both $\theta$ and $\vphi$ are proportional to the function $|x|^2$. In this section, we discuss the necessary adaptations to be implemented in our arguments in order to address this case.

Therefore, from now on we assume that 
$$
\theta(x)\,=\,|x|^2\qquad \mbox{ and } \qquad \varphi(x)\,=\,|x|^2\, ,
$$
although the discussion can be further adapted, in order to treat more general polynomial growths. In order to simplify the presentation, we also assume that $\de=\nu=0$.

\medbreak
First of all, we notice that, in view of \eqref{JfuncEnsemble}, for $J$ to be well-defined it is necessary that $|x|^2\,\rho$ belongs to $L^1$. Then, we have to assume
higher integrability on $\rho$, namely that
$$
\rho\,\in\,C\big([0,T];L^2_k(\R^d)\big)\,,\qquad\qquad\mbox{ for some }\quad k\,>\,2\,+\,\frac{d}{2}\,.
$$
This of course entails that, in \tbf{(A.1)}, one has to take $\rho_0\,\in\,H^m_k$ and $g\,\in\,L^1_T(H^m_k)$, with the same restriction $k>2+d/2$.
However, the arguments that show existence of an optimal control do not change, so that Theorem \ref{thm:existenceOptimalSolutions} still holds true.

The main changes pertain Section \ref{sec-OptimalitySystem}, starting from the Definition \ref{def:LFunction} of the functional $\LF$. First of all, let us focus on the Lagrangian multiplier $q$. On the one hand, we need it to be in some duality pairing with $\rho$: then, keeping in mind Definition \ref{def:H^m_k},
we introduce, for $(m,k)\in\N^2$, the spaces
$$
H^m_{-k}(\R^d)\,:=\,\left\{f\in H^m_{\rm loc}(\R^d)\;\big|\quad \big(1+|x|\big)^{-k}\,D^\alpha f\;\in\;L^2(\RR^d)\quad\forall\;0\leq|\alpha|\leq m\right\}\,.
$$
This space is endowed with the natural norm
$$
\|f\|_{H^m_{-k}}\,=\,\sum_{0\leq|\alpha|\leq m}\left\|\big(1+|x|\big)^{-k}\,D^\alpha f\right\|_{L^2}\,.
$$
%Alfio: one -> on
On the other hand, we still expect $q$ to solve \eqref{adjoint} to an extent, although the meaning of that equation is now no more clear, owing to the fact that $\theta$ and $\varphi$
do not belong anymore to $L^2$. To deal with both issues, we need the following lemma, whose proof
can be performed arguing as in the proof of Theorem \ref{th:weight} above, using this time the weight $\big(1+|x|\big)^{-k}$. We omit to give the details here.
\begin{lemma} \label{l:transp-neg}
Let $T>0$ and $(m,k)\in\N^2$ fixed, and let $a$ be a vector field satisfying hypotheses \eqref{hyp:data}. Moreover, assume that $q_0\in H^m_{-k}(\R^d)$ and $g\in L^1\big([0,T];H^m_{-k}(\R^d)\big)$.

Then there exists a unique solution $q\,\in\,C\bigl([0,T];H^{m}_{-k}(\RR^d)\bigr)$ to the problem
$$
\d_tq\,+\,a\cdot\nabla q\,=\,g\,,\qquad\qquad \mbox{ with }\quad q_{|t=0}\,=\,q_0\,.
$$
Moreover, there exists a ``universal'' constant $C>0$ such that the following estimate holds true for any $t\in[0,T]$:
\begin{align} \label{est:q-weight}
\standardNorm{q(t)}_{H^m_{-k}}\,\leq\,C\,\exp\left(C\,\int_0^t\left\|\nabla a(\tau)\right\|_{C^m_b}\,d\tau \right)\,
\left(\standardNorm{q_0}_{H^m_{-k}}\, +\, \int_0^t \standardNorm{g(\tau)}_{H^m_{-k}}\, d\tau \right)\,.
\end{align}
\end{lemma}

Let us come back to our optimal control problem. In view of Lemma \ref{l:transp-neg}, we can solve equation \eqref{adjoint} with $\theta$ and $\varphi$ equal to $|x|^2$,
getting a unique solution in the space $C_T(L^2_{-k})$ for any $k>2+d/2$. Let us fix, once for all, the choice\footnote{Given $z\in\R$, we denote by $[z]$ its entire part.}
$$
k_0\,=\,3\,+\,\left[\frac{d}{2}\right]\,.
$$
Then, it is easy to see that the functional $\LF$ is well-defined on the space
$$
\wtilde{\mbb{X}}_T\;:=\;\left(W^{1,1}_T(L^2_{k_0})\,\cap\,L^\infty_T(H^1_{k_0+1})\right)\;\times\;\LL^2_T\;\times\;C_T(L^2_{-k_0})\,.
$$
Of course, we also need to take $\rho_0$ and $g$ as in assumption \tbf{(A.1)}, with $m\geq 1$ and $k\geq k_0+1$.

Thereafter, we can write the optimality system \eqref{forwardEQ}-\eqref{adjoint}-\eqref{reducedGradient}, as done above. In order to characterize equation
\eqref{reducedGradient} in terms of the gradient of the reduced functional $\what J$, we need to further assume that $m\geq2$ and $k\geq k_0+2$.

\medbreak
Finally, notice that also the analysis in Section \ref{ss:unique} works similarly as above. Of course, we do not need anymore to impose assumption \tbf{(A.6)}.
On the other hand, assumption \tbf{(A.4)*} is now too strong, and we have to dismiss it.

However, we claim that it is still possible to get results analogous to Theorems \ref{thm:opt-contr_u} and \ref{thm:opt_u-constr}. More precisely, we have the following
statement for the unconstrained problem.
\begin{prop} \label{p:unique-confining}
Under assumptions \tbf{(A.1)}-\tbf{(A.2)}-\tbf{(A.5)}, suppose also that both $m\geq 2$ and $k\geq k_0+2$. In addition, take $\de=\nu=0$ in \eqref{JfuncEnsemble}, and $\theta(x)\,=\,\varphi(x)\,=\,|x|^2$.

Then, there exists at most one optimal control $u^*$ in the class $\LL^2_T$.
\end{prop}

\begin{proof}
The proof is very similar to the one to Theorem \ref{thm:opt-contr_u}, therefore we limit ourselves to put in evidence the main changes to be adopted, and to treat the most delicate points of the analysis.

As before, let $(u_1,\rho_1,q_1)$ and $(u_2,\rho_2,q_2)$ be two optimal controls with corresponding state and adjoint state. Arguing as above, we find that $\de u=u_1-u_2$ fulfils estimate \eqref{est:de-u}.
Let us now focus on the estimate of each integral appearing in that relation.

As for the former integral term, also by use of Lemma \ref{l:transp-neg}, we can write
\begin{align*}
\SPI\left|\div\big((e+x)\,\de\rho(t)\big)\,q_1(t)\right|\,dx\,&\leq\,\left\|q_1(t)\right\|_{L^2_{-k_0}}\,\left\|\div\big((e+x)\,\de\rho(t)\big)\right\|_{L^2_{k_0}} \\
&\leq\,C_3\,\left(\left\|\de\rho(t)\right\|_{L^2_{k_0}}\,+\,\left\|\nabla\de\rho(t)\right\|_{L^2_{k_0+1}}\right)\;\leq\,C_3\,\left\|\de\rho(t)\right\|_{H^1_{k_0+1}}\,.
\end{align*}
Notice that the constant $C_3$ can be expressed as follows 
\begin{equation} \label{def:C_3}
C_3\,:=\,C\,(1+T)\,\left\||x|^2\,\big(1+|x|\big)^{-k_0}\right\|_{L^2}\,\exp\Big(C\left(\left\|\div a_0\right\|_{L^1_T(L^\infty)}+\left\|u_1\right\|_{\LL^1_T}\right)\Big)\,,
\end{equation}
for a ``universal'' constant $C>0$.
At this point, the estimate for $\de\rho$ works as before, finally leading to
\begin{equation} \label{est:de-u_int1b}
\SPI\left|\div\big((e+x)\,\de\rho(t)\big) q_1(t)\right|\,dx\,\leq\,\wtilde{C}_3\,\int^t_0\big|\de u(\tau)\big|\,d\tau\,,
\end{equation}
where we have defined $\wtilde{C}_3\,=\,C_3\,K_1^{(k_0+2)}$, just depending on the data of the problem.

Next, consider the second integral in \eqref{est:de-u}. The computations are similar to the previous ones: first of all, we can estimate
\begin{align*}
\SPI\left|\div\big((e+x)\,\rho_2(t)\big)\,\de q(t)\right|\,dx\,&\leq\,\left\|\de q(t)\right\|_{L^2_{-k_0}}\,\left\|\div\big((e+x)\,\rho_2(t)\big)\right\|_{L^2_{k_0}} \\
&\leq\,\left\|\de q(t)\right\|_{L^2_{-k_0}}\,\left\|\rho_2(t)\right\|_{H^1_{k_0+1}}\;\leq\,C_4\,\left\|\de q(t)\right\|_{L^2_{-k_0}}\,,
\end{align*}
where, by Theorem \ref{th:weight} applied to equation \eqref{forwardEQ} for $\rho_2$, we obtain that
\begin{equation} \label{def:C_4}
C_4\,:=\,C\,\exp\left(C\left(\left\|\nabla a_0\right\|_{L^1_T(C^1_b)}\,+\,\left\|u_2\right\|_{\LL^1_T}\right) \right)\,
\left(\standardNorm{\rho_0}_{H^1_{k_0+1}}\, +\, \standardNorm{g}_{L^1_T(H^1_{k_0+1})} \right)\,,
\end{equation}
for a ``universal'' constant $C>0$. 
On the other hand, Lemma \ref{l:transp-neg} applied to the equation for $\de q$ gives, for a constant $C>0$, the estimate 
\begin{align*}
\left\|\de q(t)\right\|_{L^2_{-k_0}}\,&\leq\,C\,\exp\Big(C\left(\left\|\div a_0\right\|_{L^1_T(L^\infty)}\,+\,\left\|u_1\right\|_{\LL^1_T}\right)\Big)\,
\int^t_0|\de u(\tau)|\,\left\|\big(1+|x|\big)\,\nabla q_2(\tau)\right\|_{L^2_{-k_0}}\,d\tau\,. 
\end{align*}
Notice that $\left\|\big(1+|x|\big)\,\nabla q_2(\tau)\right\|_{L^2_{-k_0}}\,\leq\,\left\|\nabla q_2(\tau)\right\|_{L^2_{-k_0+1}}$. In order to bound this quantity, we can differentiate the equation
for $q_2$ with respect to $x^j$, for $1\leq j\leq d$, and get (notice that $\d_j|x|^2\,=\,2\,x^j$)
\begin{align*}
&\d_t\left(\big(1+|x|\big)^{-k_0+1}\,\d_jq_2\right)\,+\,a(t,x;u_2)\cdot\nabla\left(\big(1+|x|\big)^{-k_0+1}\,\d_jq_2\right)\,= \\
&\qquad\qquad\qquad\qquad\qquad\qquad =\,2\,x^j\,\big(1+|x|\big)^{-k_0+1}\,-\,\big(1+|x|\big)^{-k_0+1}\,\d_ja(t,x;u_2)\cdot\nabla q_2\,,
\end{align*}
with initial datum equal to $2\,x^j\,\big(1+|x|\big)^{-k_0+1}$.
Obviously, the latter term in the right-hand side can be absorbed by a Gr\"onwall argument; in addition, an easy computation shows that the former is in $L^2$. 
Therefore, by applying an $L^2$ estimate of Theorem \ref{thm:ex-u_Tr} to the previous equation implies, for a ``universal'' constant $C>0$, the following bound
$$
\left\|\nabla q_2(\tau)\right\|_{L^2_{-k_0+1}}\,\leq\,C\,\exp\Big(C\left(\left\|\nabla a_0\right\|_{L^1_T(L^\infty)}\,+\,\left\|u_2\right\|_{\LL^1_T}\right)\Big)\,(1+T)\,
\left\||x|\,(1+|x|)^{-k_0+1}\right\|_{L^2}\,.
$$
By use of this latter estimate, we finally obtain 
\begin{equation} \label{est:de-u_int2b}
\SPI\left|\div\big((e+x)\,\rho_2(t)\big)\,\de q(t)\right|\,dx\,\leq\,\wtilde{C}_4\,\int^t_0\big|\de u(\tau)\big|\,d\tau\, , 
\end{equation}
where we have defined $\wtilde{C}_4\,:=\,C_4\,\wtilde{\mc K}^{(1)}_1$ and 
\begin{equation} \label{def:K-qb}
\wtilde{\mc K}^{(1)}_1\,:=\,C\,\exp\Big(C\left(\left\|\nabla a_0\right\|_{L^1_T(C^1_b)}+\left\|u_1\right\|_{\LL^1_T}+\left\|u_2\right\|_{\LL^1_T}\right)\Big)\,(1+T)\,
\left\||x|\,(1+|x|)^{-k_0+1}\right\|_{L^2}\,.
\end{equation}

At this point, we can now insert the estimates \eqref{est:de-u_int1b} and \eqref{est:de-u_int2b} into \eqref{est:de-u}, and get, for a new constant $\mc K\,=\,\wtilde C_3+\wtilde C_4$, the relation
$$
\g\,\big|\de u(t)\big|\,\leq\,\,\mc K\,\int^t_0\big|\de u(\tau)\big|\,d\tau\,.
$$
We then conclude by an application of Gr\"onwall's lemma.
\end{proof}

The constrained case can be dealt with similarly, by the use of the new constants $C_3$, $C_4$, $\mc K^{(1)}_1$ and $K^{(k_0+2)}_1$. We omit its precise statement and presentation.
We just remark that, after defining % $\wtilde{\mc K}$ as follows
\begin{align*}
\wtilde{\mc K}\,&:=\,C\,(1+T)\,\left\|\big(1+|x|\big)^{-k_0+2}\right\|_{L^2}\,\times \\
&\qquad\qquad \times\,\exp\Big(C\left(\left\|\nabla a_0\right\|_{L^1_T(C^2_b)}+T\,\max\left\{\big|u^a\big|\,,\,\big|u^b\big|\right\}\right)\Big)\,
\left(\left\|\rho_0\right\|_{H^2_{k_0+2}}\,+\,\left\|g\right\|_{L^1_T(H^2_{k_0+2})}\right)\,,
\end{align*}
the new smallness condition reads $\wtilde{\mc K}\,T/\g\,<\,2$.

%%%%%%%%%%%%%%%%%%%%%%%%%%%%%%%%%%%%%%%%%%%%%%%%%%%%%%%%%%%%%%%%%%%%%%%%%%%%%%%%%%%%%%%%%%%%%%%%%%%%%%%%%%%%%%%%%%%%%%%%%%%%%%%%%%%%%%%%%%%%%%%%%%%%%%%%%%
\appendix

\section{Appendix -- Proof of some technical results} \label{app:appendix}

In this appendix, we collect the proof of some technical lemmas which we have used in the course of our investigation. We start with the proof to Lemma \ref{l:commutator}:
we limit ourselves to give a sketch of its proof, and refer to \cite{DiPernaLions1989} for details.
\begin{proof}[Proof of Lemma \ref{l:commutator}]
First of all, let us focus just on the dependence with respect to the space variable, and forget about the time variable in the next computations.

We start by showing that the commutator $r^j_n\,:=\,\d_j\big(\left[a,S_n\right]\,\cdot\,\big)$ is a bounded operator acting from $L^2$ into itself, uniformly in $n$. More precisely,
there exists a ``universal'' constant $C$, dependending on $a$ but independent of $j$ and (more importantly) of $n$, such that
\begin{equation} \label{est:comm-L^2}
\left\|\d_j\big(\left[a,S_n\right]\rho\big)\right\|_{L^2}\,\leq\,C\,\|\rho\|_{L^2}\qquad\qquad\forall\,\rho\,\in\,L^2\,.
\end{equation}
Indeed, straightforward computations show that
\begin{align*}
r^j_n(\rho)\,=\,\d_j\left(\big[a,S_n\big]\rho\right)\,=\,\d_j a\,S_n\rho\,+\,\Sigma^j_n\,, \qquad\qquad\mbox{ with }\qquad\Sigma^j_n\,:=\,a\,\partial_jS_n\rho\,-\,\d_jS_n(a\,\rho)\,.
\end{align*}
Of course, $\left\|\d_j a\,S_n\rho\right\|_{L^2}\,\leq\,C\,\|\nabla a\|_{L^\infty}\,\|\rho\|_{L^2}$, for a constant $C>0$ independent of $n\in\N$.
Moreover, we notice that we can write
$$
\Sigma^j_n\,=\,n^{d+1}\,\int_{\R^d}\,\d_js\big(n(x-y)\big)\,\rho(y)\,\big(a(x)\,-\,a(y)\big)\,dy\,.
$$
Hence, an application of the mean-value theorem to $a$ implies
$$
\left\|\Sigma^j_n\right\|_{L^2}\,\leq\,\left\|\nabla a\right\|_{L^\infty}\,\|\rho\|_{L^2}\,\left\|\d_js(z)\,|z|\right\|_{L^1}\,\leq\,C\,\|\rho\|_{L^2}\,,
$$
as claimed.

On the other hand, whenever $\rho$ is more regular, say $\rho\in H^1$, we can write
\begin{align*}
r^j_n(\rho)%\,:=\,\d_j\left(\big[a,S_n\big]\rho\right)
\,=\,\d_j a\,S_n\rho\,-\,S_n(\d_ja\,\rho)\,+\,a\,S_n\partial_j\rho\,-\,S_n(a\,\d_j\rho)\,=\,\d_ja\,S_n\rho\,-\,S_n(\d_j\,a\rho)\,+\,\big[a,S_n\big]\d_j\rho\,.
\end{align*}
Notice now that $\d_j a\,S_n\rho-S_n(\d_ja\,\rho)\longrightarrow 0$ in $L^2$, for $n \rightarrow \infty$: this holds true, since both terms strongly converge to $\d_ja\,\rho$ in $L^2$.
As for the last term on the right-hand side of the previous equality, we can argue similarly as above and write
\begin{align*}
\big[a,S_n\big]\d_j\rho\,&= \,n^d\int_{\RR^d} s\big(n(x-y)\big)\,\d_j\rho(y)\,\big(a(x)\,-\,a(y)\big)\,dy\,,% \\
%&= \,n^d\int_{\RR^d} s\big(n(x-y)\big)\,\d_j\rho(y)\,\left(\int_0^1 \nabla a(x+ty)dt \right)(x-y)dy \\
%&\leq \standardNorm{\nabla a}_{L^\infty(\RR^d)} \frac{1}{n} n^d \int_{\RR^d} s(n(x-y)) \d_j \rho(y)n(x-y) dy \\
%&\leq \frac{1}{n}\standardNorm{\nabla a}_{L^\infty(\RR^d)}\standardNorm{\nabla \rho}_{L^2(\RR^d)}\standardNorm{s}_{L^1(\RR^d)}.
\end{align*}
which implies, by use of the mean-value theorem again,
$$
\left\|\big[a,S_n\big]\d_j\rho\right\|_{L^2}\,\leq\,\frac{1}{n}\,\left\|\nabla a\right\|_{L^\infty}\,\left\|\nabla\rho\right\|_{L^2}\,\left\|s(z)\,|z|\right\|_{L^1}\,.
$$
Therefore, we have proven that, whenever $\rho\in H^1$, we have $r^j_n(\rho)\,\longrightarrow\,0$ strongly in $L^2$, whenever $n\ra+\infty$.

This latter property, combined with \eqref{est:comm-L^2} and a standard approximation procedure, implies  that, for any $\rho\in L^2$, we have $r^j_n(\rho)\,\longrightarrow\,0$ for $n\ra+\infty$,
in the strong topology of $L^2$.

\medbreak
Let us now consider the general case when the functions $a$ and $\rho$ also depend on the time variable, verifying the properties assumed in Lemma \ref{l:commutator}.
We have then to repeat the previous steps, keeping track of the time regularity.

We just put in evidence that, in the approximation procedure, given $\rho\,\in\,L^\infty_T(L^2)$, we can approximate it with a sequence $\big(\rho_\de\big)_{\de>0}\,\subset\,L^\infty_T(L^2)$,
with $\rho_\de\,\in\,L^\infty_T(H^1)$ for any $\de>0$ fixed, and such that $\rho_\de\,\longrightarrow\,\rho$ strongly in $L^p_T(L^2)$ for any $1\leq p<+\infty$.

This entails that, before ahead, we need to regularize $a$ with respect to time, for instance by convolution.
In particular, we take a sequence $\big(a_\eta\big)_{\eta>0}\,\subset\,L^1_T(C^1)$, with $\big(\nabla a_\eta\big)_{\eta>0}\,\subset\,L^1_T(L^\infty)$ such that $\nabla a_\eta\,\longrightarrow\,\nabla a$
in $L^1_T(L^\infty)$, and $a_\eta$ smooth (say continuous) with respect to time.

Then, after fixing some $1\leq p<+\infty$ and denoting by $q$ its conjugate exponent (namely,  $1/p\,+\,1/q\,=\,1$), we can estimate
\begin{align*}
\standardNorm{r^j_n(\rho)}_{L^1_T(L^2)}\,&\leq\,\standardNorm{\d_j \big( [a-a_\eta,S_n]\rho\big)}_{L^1_T(L^2)}\,+\,
\standardNorm{\d_j \big( [a_\eta,S_n](\rho-\rho_\delta)\big)}_{L^1_T(L^2)}\,+\,\standardNorm{\d_j\big( [a_\eta,S_n]\rho_\delta\big)}_{L^1_TL^2} \\
&\leq\,C\,\Big(\standardNorm{\nabla(a-a_\eta)}_{L^1_T(L^\infty)}\,\standardNorm{\rho}_{L^\infty_T(L^2)}\,+ \\
&\qquad\qquad\qquad\qquad\qquad +\,\standardNorm{\nabla a_\eta}_{L^q_T(L^2)}\,\standardNorm{\rho-\rho_\delta}_{L^p_T(L^2)}\,+\,
\standardNorm{\d_j \left( [a_\eta,S_n] \right)\rho_\delta}_{L^1_T(L^2)}\Big)\,.
\end{align*}
At this point, for any given $\veps>0$, we can fix first of all $\eta>0$ so that the first term in the right-hand side is smaller than $\veps$; then, in correspondence of that $\eta$,
we fix $\delta>0$ so small that also the second term is bounded by $\veps$; finally, in the last term we can pass to the limit with respect to $n$, making it smaller than $\veps$ as well.

Thus the lemma is proved.
\end{proof}

Let us now give the details of the proof to Lemma \ref{l:H^m_k}.
\begin{proof}[Proof of Lemma \ref{l:H^m_k}]
Of course, by definition of the space $H^m_k$, we have that $f\in H^m$. We only need to prove that $|x|^k\,f$ belongs to $H^m$ as well.

First of all, by an easy induction, it follows that $|x|^k\,D^\alpha f\,\in\,L^2$ for all $0\leq|\alpha|\leq m-1$.
In particular $f\,\in\,L^2_k$, i.e. $|x|^k\,f\,\in\,L^2$. Moreover, since $f\in H^m$, we gather also that $|x|^j\,D^\alpha f\,\in\,L^2$ for all $0\leq j\leq k$ and $0\leq|\alpha|\leq m-1$.

Now, fixed some $1\leq\ell\leq m$, let us take a multi-index $\alpha$ such that $|\alpha|=\ell$: we want to prove that $D^\alpha\big(|x|^k\,f\big)\,\in\,L^2$. For this, we write, by Leibniz rule,
$$
D^\alpha\left(|x|^k\,f\right)\,=\,|x|^k\,D^\alpha f\,+\,\sum_\beta D^\beta|x|^k\,D^{\alpha-\beta}f\,,
$$
where the sum is performed for all $\beta\leq\alpha$ such that $|\beta|\geq1$. The first term in the right-hand side belongs to $L^2$ by the argument here above. Moreover, one has
$$
\left|D^\beta|x|^k\right|\,\leq\,C\,|x|^{k-|\beta|}\,. %\leq\,C\,\left(1\,+\,|x|^k\right)\,,
$$
At this point, we observe that $k-|\beta|\,<\,k$, since $|\beta|\geq1$; in addition, $k-|\beta|\geq0$, owing to the fact that $|\beta|\leq|\alpha|=\ell\leq m\leq k$.
%where the second inequality holds true owing to the fact that $|\beta|\leq|\alpha|=\ell\leq m\leq k$.
Therefore, by the properties previously established, we gather that
each term $D^\beta|x|^k\,D^{\alpha-\beta}f$ of the sum also belongs to $L^2$. In turn, this implies that $D^\alpha\left(|x|^k\,f\right)\,\in\,L^2$,
for all $0\leq |\alpha|\leq m$, that is to say $|x|^k\,f\,\in\,H^m$.
\end{proof}

{\small
%\bibliography{LiouvilleTheory}{}

\begin{thebibliography}{XXX}

\bibitem{Amb_2004}
{\sc L.~Ambrosio}, {\em Transport equation and {C}auchy problem for {$BV$}
  vector fields}, Invent. Math., 158 (2004), pp.~227--260.

\bibitem{AmbrosioCrippa2008}
{\sc L.~Ambrosio and G.~Crippa}, {\em Existence, uniqueness, stability and
  differentiability properties of the flow associated to weakly differentiable
  vector fields}, in Transport equations and multi-{D} hyperbolic conservation
  laws, vol.~5 of Lect. Notes Unione Mat. Ital., Springer, Berlin, 2008,
  pp.~3--57.

\bibitem{Amb-Cr_2014}
\leavevmode\vrule height 2pt depth -1.6pt width 23pt, {\em Continuity equations
  and {ODE} flows with non-smooth velocity}, Proc. Roy. Soc. Edinburgh Sect. A,
  144 (2014), pp.~1191--1244.

\bibitem{BCD}
{\sc H.~Bahouri, J.-Y. Chemin, and R.~Danchin}, {\em Fourier analysis and
  nonlinear partial differential equations}, vol.~343 of Grundlehren der
  Mathematischen Wissenschaften [Fundamental Principles of Mathematical
  Sciences], Springer, Heidelberg, 2011.

\bibitem{Barbu12}
{\sc V.~Barbu and T.~Precupanu}, {\em Convexity and optimization in {B}anach
  spaces}, Springer Monographs in Mathematics, Springer, Dordrecht, fourth~ed.,
  2012.

\bibitem{BG-Serre_2007}
{\sc S.~Benzoni-Gavage and D.~Serre}, {\em Multidimensional hyperbolic partial
  differential equations}, Oxford Mathematical Monographs, The Clarendon Press,
  Oxford University Press, Oxford, 2007.
\newblock First-order systems and applications.

\bibitem{Brezis}
{\sc H.~Brezis}, {\em Functional analysis, {S}obolev spaces and partial
  differential equations}, Universitext, Springer, New York, 2011.

\bibitem{Brockett1997}
{\sc R.~W. Brockett}, {\em Minimum attention control}, in Proceedings of the
  36th IEEE Conference on Decision and Control, vol.~3, IEEE, 1997,
  pp.~2628--2632.

\bibitem{Brockett2007}
\leavevmode\vrule height 2pt depth -1.6pt width 23pt, {\em Optimal control of
  the {L}iouville equation}, in Proceedings of the {I}nternational {C}onference
  on {C}omplex {G}eometry and {R}elated {F}ields, vol.~39 of AMS/IP Stud. Adv.
  Math., Amer. Math. Soc., Providence, RI, 2007, pp.~23--35.

\bibitem{Brockett2012}
\leavevmode\vrule height 2pt depth -1.6pt width 23pt, {\em Notes on the control
  of the {L}iouville equation}, in Control of partial differential equations,
  vol.~2048 of Lecture Notes in Math., Springer, Heidelberg, 2012,
  pp.~101--129.

\bibitem{Candes2006}
{\sc E.~J. Cand\`es, J.~K. Romberg, and T.~Tao}, {\em Stable signal recovery
  from incomplete and inaccurate measurements}, Comm. Pure Appl. Math., 59
  (2006), pp.~1207--1223.

\bibitem{Cercignani1969}
{\sc C.~Cercignani}, {\em Mathematical methods in kinetic theory}, Plenum
  Press, New York, 1969.

\bibitem{ChoVenturiKarniadakis2015}
{\sc H.~Cho, D.~Venturi, and G.~E. Karniadakis}, {\em Numerical methods for
  high-dimensional probability density function equations}, J. Comput. Phys.,
  305 (2016), pp.~817--837.

\bibitem{CiaramellaBorzi16}
{\sc G.~Ciaramella and A.~Borz\`\i}, {\em Quantum optimal control problems with
  a sparsity cost functional}, Numer. Funct. Anal. Optim., 37 (2016),
  pp.~938--965.

\bibitem{CockshottZachariah2013}
{\sc P.~Cockshott and D.~Zachariah}, {\em Conservation laws, financial entropy
  and the eurozone crisis}, Economics, 8 (2014-5).

\bibitem{Colonna2016}
{\sc G.~Colonna and A.~D'Angola}, {\em Plasma Modeling: Methods and
  Applications}, IOP Publishing, New York, 2016.

\bibitem{Crippa_PhD}
{\sc G.~Crippa}, {\em The flow associated to weakly differentiable vector
  fields}, vol.~12 of Tesi. Scuola Normale Superiore di Pisa (Nuova Series)
  [Theses of Scuola Normale Superiore di Pisa (New Series)], Edizioni della
  Normale, Pisa, 2009.

\bibitem{DiPernaLions1989}
{\sc R.~J. DiPerna and P.-L. Lions}, {\em Ordinary differential equations,
  transport theory and {S}obolev spaces}, Invent. Math., 98 (1989),
  pp.~511--547.

\bibitem{EisemanStone1980}
{\sc P.~R. Eiseman and A.~P. Stone}, {\em Conservation laws of fluid
  dynamics---a survey}, SIAM Rev., 22 (1980), pp.~12--27.

\bibitem{Fei-No}
{\sc E.~Feireisl and A.~Novotn\'y}, {\em Singular limits in thermodynamics of
  viscous fluids}, Advances in Mathematical Fluid Mechanics, Birkh\"auser
  Verlag, Basel, 2009.

\bibitem{GodlewskiRaviart1991}
{\sc E.~Godlewski and P.-A. Raviart}, {\em Hyperbolic systems of conservation
  laws}, vol.~3/4 of Math\'ematiques \& Applications (Paris) [Mathematics and
  Applications], Ellipses, Paris, 1991.

\bibitem{Lions1971}
{\sc J.-L. Lions}, {\em Optimal control of systems governed by partial
  differential equations}, Translated from the French by S. K. Mitter. Die
  Grundlehren der mathematischen Wissenschaften, Band 170, Springer-Verlag, New
  York-Berlin, 1971.

\bibitem{Met_2008}
{\sc G.~M\'{e}tivier}, {\em Para-differential calculus and applications to the
  {C}auchy problem for nonlinear systems}, vol.~5 of Centro di Ricerca
  Matematica Ennio De Giorgi (CRM) Series, Edizioni della Normale, Pisa, 2008.

\bibitem{OceanDynamics}
{\sc D.~Olbers, J.~Willebrand, and C.~Eden}, {\em Ocean Dynamics}, Springer,
  Berlin, Heidelberg, 2012.

\bibitem{Risken1996}
{\sc H.~Risken}, {\em The {F}okker-{P}lanck equation}, vol.~18 of Springer
  Series in Synergetics, Springer-Verlag, Berlin, second~ed., 1989.
\newblock Methods of solution and applications.

\bibitem{Stadler07}
{\sc G.~Stadler}, {\em Elliptic optimal control problems with {$L^1$}-control
  cost and applications for the placement of control devices}, Comput. Optim.
  Appl., 44 (2009), pp.~159--181.

\bibitem{Troeltzsch2010}
{\sc F.~Tr\"oltzsch}, {\em Optimal control of partial differential equations},
  vol.~112 of Graduate Studies in Mathematics, American Mathematical Society,
  Providence, RI, 2010.
\newblock Theory, methods and applications, Translated from the 2005 German
  original by J\"urgen Sprekels.

\end{thebibliography}
\bibliographystyle{siam}

}

\end{document}